\documentclass[a4paper,10pt]{article}
\usepackage[utf8]{inputenc}

\usepackage{amsmath}
\usepackage{amsthm}
\usepackage{amssymb}
\usepackage{mathrsfs}
\usepackage{indentfirst}
\usepackage{graphicx}
\usepackage{floatrow}
\usepackage{float}
\usepackage{subfigure}
\usepackage{enumerate}

\usepackage{epstopdf}

\usepackage{cases}

\usepackage{multirow}

\usepackage[title]{appendix}%
\usepackage{xcolor}%
\usepackage{textcomp}%
\usepackage{booktabs}%
\usepackage{algorithm}%
\usepackage{algorithmicx}%
\usepackage{algpseudocode}%
\usepackage{listings}%

\usepackage[figuresright]{rotating}

\usepackage{threeparttable}

\usepackage{floatrow}\newfloatcommand{capbtabbox}{table}[][\FBwidth]

\allowdisplaybreaks[4]

\numberwithin{equation}{section}

\newtheorem{theorem}{Theorem}[section]
\newtheorem{definition}[theorem]{Definition} 
\newtheorem{lemma}[theorem]{Lemma}

\newtheorem{remark}[theorem]{Remark}
\newtheorem{example}[theorem]{Example}

\usepackage[center]{titlesec}

\usepackage{graphicx}

\numberwithin{equation}{section}

\usepackage[numbers,sort&compress]{natbib}

\title{The Legendre Transform of Convex Lattice Sets}
\author{Tingting He and Lin Si \thanks{Correspondence should be addressed to Lin Si: silin@bjfu.edu.cn}}

\begin{document}
\date{}
\maketitle

   \begin{abstract}
  	
   The goal of this paper is to study convex lattice sets by the discrete Legendre transform. The definition of the polar of convex lattice sets in $\mathbb{Z}^n$ is provided. 
   It is worth mentioning that the polar of convex lattice sets have the self-dual property similar to that of convex bodies.
   Some properties of convex lattice sets are established, for instance, the inclusion relation, the union and intersection on the polar of convex lattice sets. In addition, we discuss the relationship between the cross-polytope and the discrete Mahler product. 
   It states that a convex lattice set is the cross-polytope if and only if its discrete Mahler product is the smallest.
   \end{abstract}




\section{Introduction}\label{sec1}

A convex body $C$ in $n$-dimensional Euclidean space $\mathbb{E}^n$ is a compact convex subset, its polar body $C^*$ is defined by 
\begin{center}
	$C^* = \{x \in \mathbb{E}^n : \langle x,y \rangle \leq 1 $ for all $y$ in $C\},$
\end{center}
where $\langle x, y \rangle$ denotes the standard inner product of $x$ and $y$ in $\mathbb{E}^n$. 

The Blaschke-Santal$ \rm \acute{o}$ inequality is one of the most powerful results in convex geometry. It states that, 
$$V(K)V(K^*) \leq \frac{\pi ^ n}{\Gamma^{2}(1+\frac{n}{2})}$$
holds for any origin-symmetric convex body $K$, or more generally for any convex body $K$ with its centroid at the origin, with equality if and only if $K$ is an ellipsoid \cite{book5}. Here $V(\cdot)$ denotes the volume in $\mathbb{E}^n$.

Corresponding to the Blaschke-Santal$\rm \acute{o}$ inequality, the lower bound for $V(K)V(K^*)$, i.e., the Mahler conjecture, is still unsolved. 
The symmetric case of Mahler conjecture \cite{article9} asserts that, for any centrally symmetric convex body $K$, $$V(K)V(K^*) \ge \frac{4^n}{n!}?$$

For $n=2$, the conjecture has been verified by Mahler. Moreover, in \cite{article3}, 
a new proof of the Mahler conjecture in $\mathbb{E}^2$ is given by the vertex removal method.
Very recently, Iriyeh and Shibata \cite{article20} gave a new partial result of the non-symmetric case 
of Mahler conjecture for $n = 3$.
In higher dimensions, it is proved only in some very special cases, namely, for zonoids, \cite{article4}, for 1-unconditional bodies, \cite{article13}, and for the unit cube, \cite{article7}. In 1987, Bourgain and Milman \cite{article1} proved that there exists a positive constant $c$ such that 
$$V(K)V(K^*) \ge \frac{c^n}{n!},$$
which is now known as the Bourgain-Milman inequality.
The Mahler conjecture has been extensively studied, see, e.g., \cite{article21, article22, article23, article24}.

The Blaschkle-Santal$\rm \acute{o}$ inequality and Mahler conjecture consider upper and lower bounds for $V(K)V(K^*)$ in $\mathbb{E}^n$, respectively. We now turn our attention to the finite sets of integer lattice points. The geometry of integer lattice points is still an active field to research, see, e.g., \cite{article15, article14, article27, article28}.
We denote by $\mathbb{Z}^n$ the set of integer lattice in $\mathbb{E}^n$, i.e., the lattice of all points with integer coordinates. A finite set $K \subset \mathbb{Z}^n$ is a convex lattice set if $K = \mathrm{conv}(K) \cap \mathbb{Z}^n$, where $\mathrm{conv}(K)$ represents the convex hull of $K$.
Let $\#(K)\#(K_{\mathbb{Z}^n}^*)$ be the discrete Mahler product of convex lattice set $K$, where $\#(\cdot)$ denotes the number of integer lattice points and $K_{\mathbb{Z}^n}^*$ denotes the polar of convex lattice set $K$ in $\mathbb{Z}^n$ (See Definition \ref{key2}). 

Let $f:\mathbb{E}^n \rightarrow \overline{\mathbb{R}}$ be proper and convex, then the Legendre transform $f^* $ is defined by \cite{book1, book3}
\begin{align}\label{equation1}
	f^*(y) = \sup\{\langle x,y \rangle - f(x) : x \in \mathbb{E}^n \}, y \in \mathbb{E}^n,
\end{align}
where $\overline{\mathbb{R}} = \mathbb{R} \cup \{+ \infty\}$ and $\sup(\cdot)$ denotes supremum.

In \cite{article8, article12}, Murota considered the discrete Legendre transform in $\mathbb{Z}^n$. For a function $f: \mathbb{Z}^n \rightarrow \overline{\mathbb{Z}}$, 
the discrete Legendre transform $f^* : \mathbb{Z}^n \rightarrow \overline{\mathbb{Z}} $ is defined by 
\begin{align}\label{equation2}
	f^*(p) = \sup\{\langle x, p\rangle - f(x) : x \in \mathbb{Z}^n \}, p \in \mathbb{Z}^n,
\end{align}
where $\overline{\mathbb{Z}} = \mathbb{Z} \cup \{ + \infty \}$. 

For (\ref{equation2}), the geometric interpretation of the discrete Legendre transform $f^* : \mathbb{Z} \rightarrow \overline{\mathbb{Z}} $ states that when a line with slope $y$ is tangent to an integer point in $f(x)$ and intersects the $y$-axis, the opposite of the intercept of the line on the $y$-axis is obtained.

The paper is organized as follows. In Section 2, a new definition of the polar of convex lattice sets is given by the discrete Legendre transform. It should be remarked here that the new definition is essentially different from the definition of the polar of convex lattice sets in \cite{article2}. More precisely, the new definition of the polar has the self-dual property. In Section 3, more properties similar to those of the polar body of a convex body are established for the polar of convex lattice sets. Moreover, for the cross-polytope, we give a characterization by the discrete Mahler product.

\section{The polar of convex lattice sets}

We consider finite sets of integer lattice points which are necessarily full-dimensional unless indicated otherwise.
Denote by $\mathscr{K}_{\mathbb{Z}}^n$ the set of all convex lattice sets, having the origin as its interior. 
For $K \in \mathscr{K}_{\mathbb{Z}}^n$, let $K^v = \mathrm{vert}(\mathrm{conv}(K))$, where $\mathrm{vert}$ denotes vertices. That is, $K^v$ denotes all vertices of $\mathrm{conv}(K)$. A convex lattice set $K$ can be determined by vertices of its convex hull, i.e., $K = \mathrm{conv}(K^v) \cap \mathbb{Z}^n$. 

A polytope given as the convex hull of finitely many points is called a $V$-polytope \cite{book5}. Let $K \in \mathscr{K}_{\mathbb{Z}}^n$, we assume that $K^v = \{A_1, \cdots, A_{k}\}$, then
\begin{center}
	$K = \mathrm{conv}(K^v) \cap \mathbb{Z}^n = \mathrm{conv}\{A_1, \cdots, A_{k}\} \cap \mathbb{Z}^n,$
\end{center}
where $A_{r} = (a_{r 1}, \cdots, a_{r n}) \in \mathbb{Z}^n$ with $A_{r} \neq (0, \cdots, 0)$ for $1 \leq r \leq k$. 

Alternatively, if a polytope is given as the bounded intersection of finitely many closed halfspaces, then it is called an $H$-polytope \cite{book5}. More information about polytopes can be found in \cite{article31, article30}.
Let $K \in \mathscr{K}_{\mathbb{Z}}^n$. Assume that $K$ is bounded by the hyperplane $h_i$, and 
$h_i^-$ is the closed halfspace bounded by $h_i$ and containing $K$, then $K$ can be represented as
$$	K = \Big(\bigcap_{i = 1}^m h_i^- \Big) \bigcap \mathbb{Z}^n,$$
where
\begin{align}\label{eq1}
	h_i(x) = b_{i 1}x_1 + \cdots + b_{i (n - 1)}x_{n - 1} +  \beta_i, \  i = 1,\cdots, m,
\end{align}
with $b_{i j}, \beta_i \in \mathbb{Z}$ for $j = 1, \cdots, n - 1$.

For (\ref{eq1}) and $p \in \mathbb{Z}^{n - 1}$, by (\ref{equation2}), we get that
\begin{align*}
	h_i^*(p) & = \sup\{\langle x, p\rangle - h_i(x) : x \in \mathbb{Z}^{n - 1} \} \\
	& = \sup\{\langle x, p\rangle - (b_{i 1}x_1 + \cdots + b_{i (n - 1)}x_{n - 1} +  \beta_i) : x \in \mathbb{Z}^{n - 1} \},
\end{align*}
for $i = 1, \cdots, m.$

Fix $p_i = (p_{i 1}, \cdots, p_{i (n-1)})$, $i = 1, \cdots, m$. Consider 
\begin{align*}
	h_i^*(p_i) = \sup\{ \langle  x, p_i \rangle - (b_{i 1}x_1 + \cdots + b_{i (n - 1)}x_{n - 1} +  \beta_i) : x \in \mathbb{Z}^{n - 1} \},
\end{align*}
for $i = 1, \cdots, m.$
Then, we have
\begin{align*}
	\frac{\partial }{\partial x_j} \Big(\sum_{j = 1}^{n-1}p_{i j} x_j - h_i(x) \Big) = p_{i j} - \frac{\partial h_i(x)}{\partial x_j} = 0,
\end{align*}
and hence
\begin{align}\label{eq2}
	p_{ij} = \frac{\partial h_i(x)}{\partial x_j}, 
\end{align}
for $i = 1, \cdots, m$, $j = 1, \cdots, n - 1$.

Combining (\ref{eq1}) and (\ref{eq2}), we see that $p = p_i = (b_{i 1}, \cdots ,b_{i (n - 1)}),$ so
\begin{equation*}
	\begin{aligned}
		h_i^*(p) & = \sup\{\langle x, p_i \rangle - h_i(x) \} \\
		& = \sup\{(p_{i 1} - b_{i 1})x_1 + \cdots + (p_{i (n - 1)} - b_{i (n - 1)})x_{n - 1} - \beta_i\}  \\
		& = -\beta_i,
	\end{aligned}
\end{equation*}
for $i = 1, \cdots, m$.

If $p \neq (b_{i 1}, \cdots, b_{i (n - 1)})$, then
\begin{align*}
	h_i^*(p)  = \sup\{\langle x, p \rangle - h_i(x) \} 
	= +\infty,
\end{align*}
for $i = 1, \cdots, m$.

By the above argument, the discrete Legendre transform $h_i^*(p) : \mathbb{Z}^{n - 1} \rightarrow \overline{\mathbb{Z}}$ is given by
$$ 
h_i^*(p) =
\left\{
\begin{aligned}
	-\beta_i, \ \quad & p = (b_{i 1}, \cdots, b_{i (n - 1)}), \\
	+ \infty, \ \quad & {\rm{otherwise}},
\end{aligned}
\right.
$$
for $i = 1, \cdots, m$.

Clearly, $h_i^*(p)$ is a proper convex function in $\mathbb{Z}^{n - 1}$, which is characterized in \cite{article19}.

For $p = p_{i} = (b_{i 1}, \cdots, b_{i (n - 1)})$, we get that $h_i^* (p_i)= -\beta_i$. Let 
\begin{align}\label{eq21}
	B_i = (p_{i 1}, \cdots, p_{i (n - 1)}, h_i^*(p_{i})) = (b_{i1}, \cdots, b_{i (n - 1)}, -\beta_i) 
\end{align}
with $b_{i j}, \beta_i \in \mathbb{Z}$ for $i = 1, \cdots, m$, $j = 1, \cdots, n - 1$. 
Denote by $K_0 = \{ B_1, \cdots, B_m\}$.
Then, we let
\begin{align}\label{eq9}
	K_L = \mathrm{conv}(\bigcup_{i = 1}^{m} B_i) \cap \mathbb{Z}^n,
\end{align}
and therefore, we have $K_L^v \subseteq K_0 $.

For $A_{r} \in K^v$ and $B_i \in K_L$, let $\lambda$ satisfy the following conditions:
\begin{align}\label{a5}
	\lambda = \underset{1 \leq r \leq k}{\max} \Big\{ \underset{1 \leq i \leq m}{\min} \Big\{  |\lambda_i| \Big| \overrightarrow{OA_{r}} \cdot \lambda_i \overrightarrow{OB_i} \leq 1 \ {\rm{for}} \  i = 1, \cdots, m, \ r = 1, \cdots, k \Big\} \Big\},
\end{align} 
where $\overrightarrow{ \cdot }$ denotes the vector. 

Let $\overrightarrow{OC_i} = \lambda \overrightarrow{OB_i}$, $i = 1, \cdots, m$.
Denote by $\mathbb{Q}^n$ the set of $n$-tuple arrays of rational numbers, i.e., the set of points with rational coordinates. Then, we give a definition of the polar of a convex lattice set $K$ in $\mathbb{Q}^n$.

\begin{definition}\label{key1}
	The polar of a convex lattice set $K$ in $\mathbb{Q}^n$ is defined by
	$$K_{\mathbb{Q}^n}^* = \mathrm{conv} (\bigcup_{i = 1}^{m} C_i ).$$
\end{definition}

Following (\ref{eq21}) and Definition \ref{key1}, we can obtain
\begin{align}\label{eq22}
	K_{\mathbb{Q}^n}^* = \Big \{ \sum_{i = 1}^m \mu_i C_i \Big| \sum_{i = 1}^m \mu_i = 1, \mu_i \ge 0 \ {\rm{for}} \ 1 \leq i \leq m \Big \},
\end{align}
where $C_i = (\lambda b_{i1}, \cdots, \lambda b_{i (n - 1)}, -\lambda \beta_i)$ for $i = 1, \cdots, m$.

For $K_{\mathbb{Q}^n}^*$, the transform $ \mathcal{T}$ is defined by
\begin{align}\label{eq3}
	\mathcal{T}(K_{\mathbb{Q}^n}^*) = \frac{1}{\lambda}K_{\mathbb{Q}^n}^*.
\end{align}
Together with (\ref{a5}) and the fact that $K \in \mathscr{K}_{\mathbb{Z}}^n$, it is obvious that $\lambda \neq 0$. 

We define the polar of a convex lattice set $K$ in $\mathbb{Z}^n$ as follows.

\begin{definition}\label{key2}
	The polar of a convex lattice set $K$ in $\mathbb{Z}^n$ is defined by
	$$K_{\mathbb{Z}^n}^* = \mathrm{conv}(\mathcal{T}(K_{\mathbb{Q}^n}^*)) \cap \mathbb{Z}^n.$$
\end{definition}

\begin{example}\label{example1}
	For $K \in \mathscr{K}_{\mathbb{Z}}^2$ and $K^v = \{(-1,1),(0,2),(2,0),$ $(0,-2)\}$, by the discrete Legendre transform {\rm(}\ref{equation2}{\rm)}, Definition \ref{key1} and Definition \ref{key2}, we have
	$K_L^v =\{(1,-2),(-1,-2),(1,2),(-3,2)\},$ 
	$(K_{\mathbb{Q}^2}^*)^v = \{(\frac{1}{5},-\frac{2}{5}),(-\frac{1}{5},-\frac{2}{5}),(\frac{1}{5},\frac{2}{5}),$ $(-\frac{3}{5},\frac{2}{5})\},$ and 
	$(K_{\mathbb{Z}^2}^*)^v = \{(1,-2),(-1,-2),(1,2),(-3,2)\}$. 
	See Fig. \ref{Fig:figure2}.
	
	\begin{figure}[H]                       
		\includegraphics[width=1.0\textwidth]{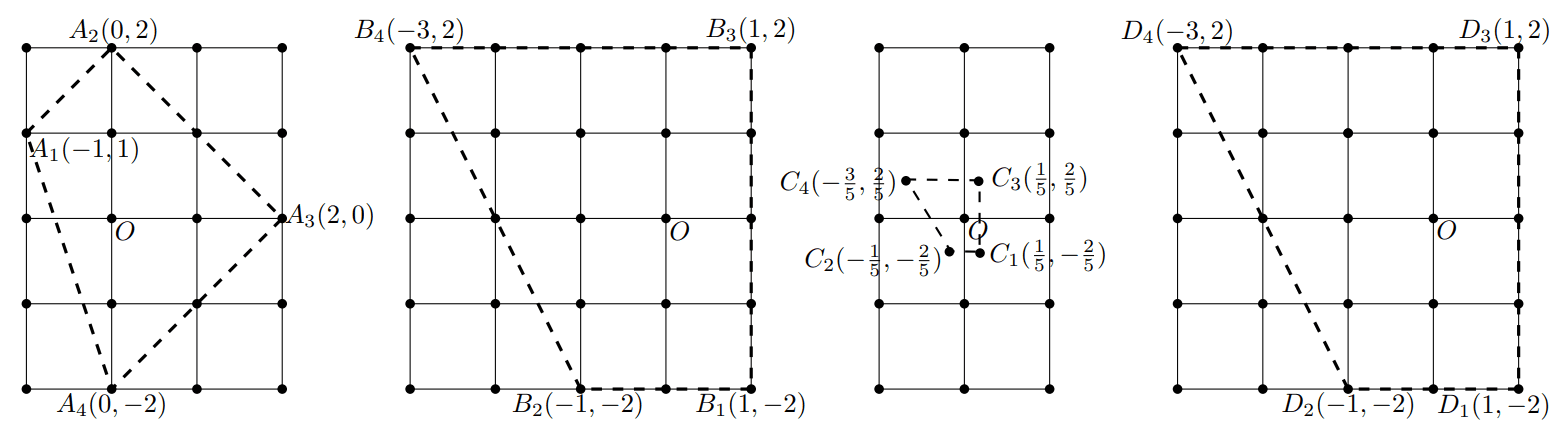}
		\caption{\label{Fig:figure2}$K$, \ $K_L$, \ $K_{\mathbb{Q}^2}^*$, \ $K_{\mathbb{Z}^2}^*$ (from left)}
	\end{figure}
\end{example}

\begin{remark}
	We abbreviate the discrete Legendre transform by $DLT$,
	and then the relations between $K$, $K_L$, $K_{\mathbb{Q}^n}^*$ and $K_{\mathbb{Z}^n}^*$ can be represented as follows
	$$
	K
	\xrightarrow{DLT}
	K_L 
	\xrightarrow{ \lambda  } 
	K_{\mathbb{Q}^n}^*
	\xrightarrow{ \mathcal{T}(K_{\mathbb{Q}^n}^*) }
	K_{\mathbb{Z}^n}^*.
	$$
\end{remark}

In \cite{article8}, for $f:\mathbb{Z}^n \rightarrow \overline{\mathbb{Z}}$, $f^*$ is an integer-valued function and $(f^*)^* = f$. A similar property also holds for a convex body $K$, namely $(K^*)^* = K$. In what follows, we consider the self-dual property of convex lattice sets in $\mathbb{Z}^n$.

\begin{theorem}\label{key3}
	If $K \in \mathscr{K}_{\mathbb{Z}}^n$ and $K$ is bounded by the hyperplane $h_i$, then
	$$(K_{\mathbb{Z}^{n}}^\ast)_{\mathbb{Z}^{n}}^\ast = K,$$
	where
	$$h_i(x) = b_{i 1}x_1 + \cdots + b_{i (n-1)}x_{n-1} +  \beta_i,$$ 
	with $b_{i j}, \beta_i \in \mathbb{Z}$ for $i = 1, \cdots, m$, $j = 1, \cdots, n - 1$.
\end{theorem}

\begin{proof}
	Let the vertices of convex lattice set $K$ be $K^v = \{A_1, \cdots, A_{k}\}$, where $A_{r} = (a_{r 1}, \cdots, a_{r n}) \in \mathbb{Z}^n $ with $A_{r} \neq (0, \cdots, 0)$ for $1 \leq r \leq k$.
	
	Combining (\ref{eq21}) and (\ref{eq9}), we have
	\begin{equation}\label{tend1}
		\begin{aligned}
			K_L = \mathrm{conv}(\bigcup_{i = 1}^{m} (b_{i 1}, \cdots, b_{i (n - 1)}, -\beta_i)) \cap \mathbb{Z}^n.
		\end{aligned}
	\end{equation}
	
	It follows from (\ref{a5}), (\ref{eq22}) and Definition \ref{key1} that there exists a suitable $\lambda_1(\lambda_1 \neq 0)$ such that
	\begin{align}\label{a7}
		K_{\mathbb{Q}^n}^* = \mathrm{conv}(\bigcup_{i = 1}^{m} \lambda_1 (b_{i 1}, \cdots, b_{i (n - 1)}, -\beta_i)).
	\end{align}
	
	We let 
	\begin{align*}
		K_{\mathbb{Q}^n}^* = \{ (q_{1 1}, \cdots, q_{1 n} ), \cdots,(q_{s 1}, \cdots, q_{s n} ), \cdots, (q_{l 1}, \cdots, q_{l n} )\},
	\end{align*}
	where $q_{st} \in \mathbb{Q}$ for $s = 1, \cdots, l$, $t = 1, \cdots, n$.
	
	By Definition \ref{key2}, we can get that
	\begin{align}\label{a8}
		K_{\mathbb{Z}^n}^* = \mathrm{conv} \Big \{ \Big (\frac{q_{1 1}}{\lambda_1 }, \cdots, \frac{q_{1 n}}{\lambda_1 } \Big), \cdots, \Big (\frac{q_{l 1}}{\lambda_1 }, \cdots, \frac{q_{l n}}{\lambda_1 } \Big) \Big \} \bigcap \mathbb{Z}^n.
	\end{align}
	Together with (\ref{a7}) and (\ref{a8}), we have
	\begin{align}\label{tend2}
		K_{\mathbb{Z}^n}^* = \mathrm{conv}\{ (b_{1 1}, \cdots, b_{1 (n - 1)}, -\beta_1), \cdots, (b_{m 1}, \cdots, b_{m (n - 1)}, -\beta_m)\} \cap \mathbb{Z}^n.
	\end{align} 
	
	Obviously, (\ref{tend1}) and (\ref{tend2}) yield that $K_L$ = $K_{\mathbb{Z}^n}^*$.
	Let 
	\begin{align*}
		K_0 = \{ (b_{1 1}, \cdots, b_{1 (n - 1)}, -\beta_1), \cdots, (b_{m 1}, \cdots, b_{m (n - 1)}, -\beta_m) \},
	\end{align*}
	and then, we get that any point in $K_0$ satisfies $h_i^*(p) = -\beta_i$, where $p = p_i = (b_{i 1}, \cdots, b_{i (n - 1)})$ for $i = 1, \cdots, m$. 
	If $p \neq (b_{i 1}, \cdots, b_{i (n - 1)}),$ then $h_i^*(p) = + \infty$.
	
	Therefore, $h_i^*(p) : \mathbb{Z}^{n - 1} \rightarrow \overline{\mathbb{Z}} $ is given by
	$$ 
	h_i^*(p) =\left\{
	\begin{aligned}
		-\beta_1,              \ \quad & p = (b_{i 1}, \cdots, b_{i (n - 1)}), \\
		+ \infty,              \ \quad & {\rm{otherwise,}}
	\end{aligned}
	\right.
	$$
	for $i = 1, \cdots, m$. Obviously, $h_i^*(p)$ is a proper convex function in $\mathbb{Z}^{n - 1}$.
	Then, for $h_i^*(p)$, by (\ref{equation2}), we have
	\begin{equation}\label{a9}
		\begin{aligned}
			((h_i(x))^*)^*  & = \sup\{\langle x, p \rangle - h_i^*(p ) : p \in \mathbb{Z}^{n - 1}\},x \in \mathbb{Z}^{n - 1} \\
			& = h_i(x),
		\end{aligned}
	\end{equation}
	for $i = 1, \cdots, m$.
	
	Thus, by (\ref{tend2}) and (\ref{a9}), we see that $(K_{\mathbb{Z}^n}^*)_L$ is also bounded by the hyperplanes $h_i$ for $i = 1, \cdots, m$. 
	Then, we obtain
	\begin{equation*}
		\begin{aligned}
			(K_{\mathbb{Z}^n}^*)_L^v = \{(a_{1 1}, \cdots ,a_{1 n}), \cdots, (a_{k 1},  \cdots, a_{k n})\},
		\end{aligned}
	\end{equation*}
	that is,
	\begin{align*}
		(K_{\mathbb{Z}^n}^*)_L = \mathrm{conv}( \bigcup_{r = 1}^{k} A_{r} ) \cap \mathbb{Z}^n,
	\end{align*}
	where $A_{r} = (a_{r 1}, \cdots, a_{r n})$ for $r = 1, \cdots, k$.
	
	Moreover, it follows from (\ref{a5}) and Definition \ref{key1} that there exists a suitable $\lambda_2(\lambda_2 \neq 0)$ such that
	\begin{equation*}
		\begin{aligned}
			(K_{\mathbb{Z}^n}^*)_{\mathbb{Q}^n}^* = \mathrm{conv}( \bigcup_{r = 1}^{k} \lambda_2 A_{r}).
		\end{aligned}
	\end{equation*}
	Therefore, by (\ref{eq3}) and Definition \ref{key2}, we get that
	\begin{equation*}
		\begin{aligned}
			(K_{\mathbb{Z}^n}^*)_{\mathbb{Z}^n}^* = \mathrm{conv}\{ (a_{11}, \cdots ,a_{1n}), \dots, (a_{k 1}, \cdots ,a_{k n})\} \cap \mathbb{Z}^n.
		\end{aligned}
	\end{equation*}
	
	Then, we have
	$$(K_{\mathbb{Z}^{n}}^*)_{\mathbb{Z}^{n}}^* = K.$$ 
\end{proof}

\begin{remark}
	The relations between the geometric objects in Theorem \ref{key3} can be represented as follows
	\begin{equation*}	
		\begin{aligned}
			K 
			\xrightarrow{DLT} 
			K_L
			\xrightarrow{ \lambda_1  } 
			K_{\mathbb{Q}^n}^*  
			\xrightarrow{ \mathcal{T}(K_{\mathbb{Q}^n}^*) } 
			K_{\mathbb{Z}^n}^*  
			\xrightarrow{DLT} 
			(K_{\mathbb{Z}^n}^*)_L
			\xrightarrow{\lambda_2  } 
			(K_{\mathbb{Z}^n}^*)_{\mathbb{Q}^n}^*
			\xrightarrow{\mathcal{T}(K_{\mathbb{Q}^n}^*) } 
			K \nonumber.
		\end{aligned}
	\end{equation*}
\end{remark}

\begin{remark}
	In Theorem \ref{key3}, if the coefficients of the hyperplane $h_i$ satisfy $b_{i j} \in \mathbb{Q}$, then the definition of the polar of a convex lattice set $K$ can be given by the Legendre transform, and hence $(K_{\mathbb{Z}^{n}}^*)_{\mathbb{Z}^{n}}^* = K$ can also be obtained. The proof is similar to that of Theorem \ref{key3}. 
	More precisely, the two results do not contain each other except for a common part. 
	\footnote{For $K_1 \in \mathscr{K}_{\mathbb{Z}}^2$ and $K_1^v = \{(0,6), (2,2), (0,-2), (-2,2)\}$, we have $(K_1)_{\mathbb{Z}^{2}}^* = \mathrm{conv}\{(-2,-6), (2,2), (-2,2), (2,-6)\} \cap \mathbb{Z}^2$ by the discrete Legendre transform, while  $((K_1)_{\mathbb{Z}^{2}}^*)_{\mathbb{Z}^{2}}^* = K_1$ cannot be obtained by the Legendre transform. For $K_2 \in \mathscr{K}_{\mathbb{Z}}^2$ and $K_2^v = \{(4,3),$ $(1,-3), (-2,-1), (1,2)\}$, by the Legendre transform, we conclude that $(K_2)_{\mathbb{Z}^{2}}^* = \mathrm{conv}\{(3,-3), (1,-5), (6,15), (-2,7)\} \cap \mathbb{Z}^2$, while $((K_2)_{\mathbb{Z}^{2}}^*)_{\mathbb{Z}^{2}}^* = K_2$ cannot be obtained $K_2$ by the discrete Legendre transform.}
\end{remark}

\setcounter{table}{0}
\setcounter{figure}{0}

\section{The properties of convex lattice sets and discrete Mahler product}

Assume that
\begin{equation}\label{a27}
	\begin{aligned}
		h_{\beta_1 i}(x) & = b_{i 1}x_1 + \cdots + b_{i (n-1)}x_{n-1} + \beta_1, \ {\rm{for}} \ \beta_1 > 0, i = 1, \cdots, m_1, \\
		h_{\beta_2 l}(x) & = d_{l 1}x_1 + \cdots + d_{l (n-1)}x_{n-1} + \beta_2, \ {\rm{for}} \ \beta_2 < 0, l = 1, \cdots, m_2,
	\end{aligned}
\end{equation}
with $b_{i j}, d_{l j} ,\beta_1, \beta_2 \in \mathbb{Z}$ for $j = 1, \cdots, n - 1$.

In what follows, we turn our attention to a special subclass of  $\mathscr{K}_{\mathbb{Z}}^n$, consisting of all convex lattice sets in $\mathbb{Z}^n$, generated by the intersection of hyperplanes (\ref{a27}), that include the origin as its interior, that is, 
\begin{align*}
	\Big \{ \Big(\bigcap_{i = 1}^{m_1} h_{\beta_1 i}^- \bigcap \{x_n \geq 0\}  \Big) \bigcup \Big(\bigcap_{l = 1}^{m_2} h_{\beta_2 l}^- \bigcap \{x_n < 0\} \Big) \Big \} \bigcap \mathbb{Z}^n,
\end{align*}
which will be denoted by $\mathscr{C}_{\mathbb{Z}}^n$.

To prove the inclusion relation for the polar of convex lattice sets, we need the following lemma.

\begin{lemma}[\cite{book4}]\label{q1}
	If $K,L \subset \mathbb{Z}^n$ and $K \subseteq L$, then 
	\footnote{For Lemma \ref{q1}, the original form of this lemma \cite{book4} states that if $K,L \subset \mathbb{E}^n$ and $K \subseteq L$, then $\mathrm{conv}(K) \subseteq \mathrm{conv}(L)$.}
	$$\mathrm{conv}(K) \cap \mathbb{Z}^n \subseteq \mathrm{conv}(L) \cap \mathbb{Z}^n.$$ 
\end{lemma}

We omit the proof of Lemma \ref{q1}, since it is easy to get. 
Next, we introduce a lemma for union on sets.

\begin{lemma}[\cite{book1}]\label{b2}
	If $K,L \subset \mathbb{Z}^n$, then
	\footnote{A result similar to Lemma \ref{b2} also holds in {\rm{\cite{book1}}}. It states that if $ K,L \subset \mathbb{E}^n$, then $\mathrm{conv}( \mathrm{conv}(K) \cup \mathrm{conv}(L) ) = \mathrm{conv}( K \cup L ).$} 
	$$\mathrm{conv}\{(\mathrm{conv}(K) \cap \mathbb{Z}^n) \cup (\mathrm{conv}(L) \cap \mathbb{Z}^n )\}\cap \mathbb{Z}^n = \mathrm{conv}(K \cup L) \cap \mathbb{Z}^n.$$
\end{lemma}

\begin{proof}
	Obviously, $\mathrm{conv}(K) \cap \mathbb{Z}^n \subset \mathrm{conv}(K \cup L) \cap \mathbb{Z}^n$ and $\mathrm{conv}(L) \cap \mathbb{Z}^n \subset \mathrm{conv}(K \cup L) \cap \mathbb{Z}^n$. 
	Then, this means that 
	$$(\mathrm{conv}(K) \cap \mathbb{Z}^n )\cup (\mathrm{conv}(L) \cap \mathbb{Z}^n) \subset \mathrm{conv}(K \cup L) \cap \mathbb{Z}^n.$$
	Consequently, we have
	\begin{align*}
		& \quad \ \mathrm{conv}\{(\mathrm{conv}(K) \cap \mathbb{Z}^n )\cup ( \mathrm{conv}(L) \cap \mathbb{Z}^n)\} \cap \mathbb{Z}^n \\
		& \subset \mathrm{conv}\{ \mathrm{conv}(K \cup L) \cap \mathbb{Z}^n \} \cap \mathbb{Z}^n \\
		& = \mathrm{conv}(K \cup L) \cap \mathbb{Z}^n.
	\end{align*}
	
	On the other hand, it remains to prove opposite inclusion
	$$\mathrm{conv}(K \cup L) \cap \mathbb{Z}^n \subset \mathrm{conv}\{(\mathrm{conv}(K) \cap \mathbb{Z}^n )\cup (\mathrm{conv}(L) \cap \mathbb{Z}^n)\} \cap \mathbb{Z}^n.$$
	
	It follows from $K \subset \mathrm{conv}(K) \cap \mathbb{Z}^n$ and $L \subset \mathrm{conv}(L) \cap \mathbb{Z}^n$ that 
	$$K \cup L \subset ({\rm{conv}}(K) \cap \mathbb{Z}^n )\cup ({\rm{conv}}(L) \cap \mathbb{Z}^n).$$
	Then, we get that
	\begin{align*}
		\mathrm{conv}(K \cup L) \cap \mathbb{Z}^n \subset \mathrm{conv}\{(\mathrm{conv}(K) \cap \mathbb{Z}^n )\cup (\mathrm{conv}(L) \cap \mathbb{Z}^n)\} \cap \mathbb{Z}^n.
	\end{align*}
	
	Therefore, we have
	\begin{align*}
		\mathrm{conv}\{(\mathrm{conv}(K) \cap \mathbb{Z}^n) \cup (\mathrm{conv}(L) \cap \mathbb{Z}^n )\}\cap \mathbb{Z}^n = \mathrm{conv}(K \cup L) \cap \mathbb{Z}^n.
	\end{align*}
\end{proof}

We now prove the inclusion relation for the polar of convex lattice sets.

\begin{theorem}\label{key4}
	If $K, L \in \mathscr{C}_{\mathbb{Z}}^n$, then 
	$K \subseteq L$ implies $L_{\mathbb{Z}^n}^* \subseteq  K_{\mathbb{Z}^n}^*.$
\end{theorem}

\begin{proof}
	By assumption, $K$ is bounded by the following hyperplanes:
	\begin{equation}\label{eq4}
		\begin{aligned}
			h_{\beta_1 i}(x) & = b_{i 1}x_1 + \cdots + b_{i (n-1)}x_{n-1} + \beta_1, \ {\rm{for}} \ \beta_1 > 0, i = 1, \cdots, m_1,\\
			h_{\beta_2 l}(x) & = d_{l 1}x_1 + \cdots + d_{l (n-1)}x_{n-1} + \beta_2, \ {\rm{for}} \ \beta_2 < 0, l = 1, \cdots, m_2,
		\end{aligned}
	\end{equation}
	with $b_{i j}, d_{l j}, \beta_1, \beta_2 \in \mathbb{Z}$ for $j = 1, \cdots, n - 1$.
	
	Similarly, $L$ is bounded by the following hyperplanes:
	\begin{equation}\label{eq5}
		\begin{aligned}
			g_{\beta_1 u}(x) & = c_{u 1}x_1 + \cdots + c_{u (n - 1)}x_{n - 1} + \beta_1, \ {\rm{for}} \ \beta_1 > 0, u = 1, \cdots, s_1,\\
			g_{\beta_2 v}(x) & = e_{v 1}x_1 + \cdots + e_{v (n - 1)}x_{n - 1} + \beta_2, \ {\rm{for}} \ \beta_2 < 0, v = 1, \cdots, s_2,
		\end{aligned}
	\end{equation}
	with $c_{u j}, e_{v j}, \beta_1, \beta_2 \in \mathbb{Z}$ for $j = 1, \cdots, n - 1$.
	
	It follows from (\ref{eq9}), Definition \ref{key1} and Definition \ref{key2} that
	\begin{align}\label{conv1}
		K_{\mathbb{Z}^n}^* 
		= \mathrm{conv}  \Big\{  \Big(\bigcup_{i = 1}^{m_1} (b_{i 1}, \cdots, b_{i (n - 1)}, -\beta_1 ) \Big) \bigcup 
		\Big(\bigcup_{l = 1}^{m_2} (d_{l 1}, \cdots, d_{l (n - 1)}, -\beta_2 ) \Big) \Big\} \bigcap \mathbb{Z}^n,
	\end{align}
	\begin{align}\label{conv2}
		L_{\mathbb{Z}^n}^* 
		= \mathrm{conv} \Big \{ \Big (\bigcup_{u = 1}^{s_1} (c_{u 1}, \cdots, c_{u (n - 1)}, -\beta_1 ) \Big) \bigcup 
		\Big(\bigcup_{v = 1}^{s_2} (e_{v 1}, \cdots, e_{v (n - 1)}, -\beta_2 ) \Big) \Big\} \bigcap \mathbb{Z}^n.
	\end{align}
	
	For convenience, let
	\begin{equation}\label{a25}
		\begin{aligned}
			K_1 & =  \{ (b_{1 1}, \cdots, b_{1 (n - 1)}, -\beta_1), \cdots, (b_{m_1 1}, \cdots, b_{m_1 (n - 1)}, -\beta_1), \\
			& \ \quad
			(d_{1 1}, \cdots, d_{1 (n - 1)}, -\beta_2), \cdots, (d_{m_2 1}, \cdots, d_{m_2 (n - 1)}, -\beta_2) \},
		\end{aligned}
	\end{equation}
	\begin{equation}\label{a26}
		\begin{aligned}
			L_1 & =  \{ (c_{1 1}, \cdots, c_{1 (n - 1)}, -\beta_1), \cdots, (c_{s_1 1}, \cdots, c_{s_1 (n - 1)}, -\beta_1), \\
			& \ \quad
			(e_{1 1}, \cdots, e_{1 (n - 1)}, -\beta_2), \cdots, (e_{s_2 1}, \cdots, e_{s_2 (n - 1)}, -\beta_2) \},
		\end{aligned}
	\end{equation}
	and then, we have
	\begin{align}\label{a24}
		( K_{\mathbb{Z}^n}^*)^v \subseteq K_1, \ (L_{\mathbb{Z}^n}^*)^v \subseteq L_1.
	\end{align}
	
	Obviously, for $K_1$ and $L_1$, we get that both $K_1$ and $L_1$ concentrate on $x_n = -\beta_1$ (or $x_n = -\beta_2$), and therefore, by (\ref{a25}), (\ref{a26}) and (\ref{a24}), we know that both $(K_{\mathbb{Z}^n}^*)^v$ and $(L_{\mathbb{Z}^n}^*)^v$ are also concentrated on $x_n = -\beta_1$ (or $x_n = -\beta_2$), where $x_n$ denotes the $n$-th coordinate.
	
	According to sign characteristics of coordinates, we divide $\mathbb{E}^n$ into the following $2^n$ blocks:
	\begin{gather*}
		I_1 = \{(x_1, \cdots, x_n) \in \mathbb{Z}^n : x_1 \geq 0, x_2 \geq 0, x_3 \geq 0, \cdots,x_{n - 1} \geq 0, x_n \geq 0\}, \\
		I_2 = \{(x_1, \cdots, x_n) \in \mathbb{Z}^n : x_1 \leq 0, x_2 \geq 0, x_3 \geq 0, \cdots,x_{n - 1} \geq 0, x_n \geq 0\}, \\
		I_3 = \{(x_1, \cdots, x_n) \in \mathbb{Z}^n : x_1 \leq 0, x_2 \geq 0, x_3 \leq 0, \cdots,x_{n - 1} \geq 0, x_n \geq 0\}, \\
		\vdots \\
		I_{2^{n - 2} + 1} = \{(x_1, \cdots, x_n) \in \mathbb{Z}^n : x_1 \leq 0, x_2 \leq 0, x_3 \leq 0, \cdots,x_{n - 1} \leq 0, x_n \geq 0\}, \\
		I_{2^{n - 2} + 2} = \{(x_1, \cdots, x_n) \in \mathbb{Z}^n : x_1 \geq 0, x_2 \leq 0, x_3 \geq 0, \cdots,x_{n - 1} \geq 0, x_n \geq 0\}, \\
		I_{2^{n - 2} + 3} = \{(x_1, \cdots, x_n) \in \mathbb{Z}^n : x_1 \geq 0, x_2 \geq 0, x_3 \leq 0, \cdots,x_{n - 1} \geq 0, x_n \geq 0\}, \\
		\vdots \\
		I_{2^{n - 1}} = \{(x_1, \cdots, x_n) \in \mathbb{Z}^n : x_1 \geq 0, x_2 \leq 0, x_3 \leq 0, \cdots,x_{n - 1} \leq 0, x_n \geq 0\},\\
		I_{2^{n - 1} + 1} = \{(x_1, \cdots, x_n) \in \mathbb{Z}^n : x_1 \geq 0, x_2 \geq 0, x_3 \geq 0, \cdots,x_{n - 1} \geq 0, x_n \leq 0\},\\
		I_{2^{n - 1} + 2} = \{(x_1, \cdots, x_n) \in \mathbb{Z}^n : x_1 \leq 0, x_2 \geq 0, x_3 \geq 0, \cdots,x_{n - 1} \geq 0, x_n \leq 0\},\\
		I_{2^{n - 1} + 3} = \{(x_1, \cdots, x_n) \in \mathbb{Z}^n : x_1 \leq 0, x_2 \geq 0, x_3 \leq 0, \cdots,x_{n - 1} \geq 0, x_n \leq 0\},\\
		\vdots \\
		I_{3 \cdot 2^{n - 2} + 1} = \{(x_1, \cdots, x_n) \in \mathbb{Z}^n : x_1 \leq 0, x_2 \leq 0, x_3 \leq 0, \cdots,x_{n - 1} \leq 0, x_n \leq 0\}, \\
		I_{3 \cdot 2^{n - 2} + 2} = \{(x_1, \cdots, x_n) \in \mathbb{Z}^n : x_1 \geq 0, x_2 \leq 0, x_3 \geq 0, \cdots,x_{n - 1} \geq 0, x_n \leq 0\}, \\
		I_{3 \cdot 2^{n - 2} + 3} = \{(x_1, \cdots, x_n) \in \mathbb{Z}^n : x_1 \geq 0, x_2 \geq 0, x_3 \leq 0, \cdots,x_{n - 1} \geq 0, x_n \leq 0\}, \\
		\vdots \\
		I_{2^n}  = \{(x_1, \cdots, x_n) \in \mathbb{Z}^n : x_1 \geq 0, x_2 \leq 0, x_3 \leq 0, \cdots, x_{n - 1} \leq 0, x_n \leq 0\}. 
	\end{gather*}
	
	Together with $K \subseteq L$ and the fact that $K, L \in \mathscr{C}_{\mathbb{Z}}^n$, we have 
	\begin{align}\label{eq12}		
		|b_{i j}| \geq |c_{u j}|
	\end{align}
	in $I_{t}$ for $1 \leq t \leq 2^{n - 1}$ and 
	\begin{align}\label{eq13}
		|d_{l j}| \geq |e_{v j}|
	\end{align}
	in $I_{2^{n - 1} + t}$ for $1 \leq t \leq 2^{n - 1}$.
	
	Let $M$ be the intersection of $x_n = -\beta_1$ and $x_n$-axis and $N$ be the intersection of $x_n = -\beta_2$ and $x_n$-axis, then the coordinates of $M$ and $N$ are $(0, \cdots, 0, -\beta_1)$ and $(0, \cdots, 0, -\beta_2)$, respectively.
	
	Let $P \in K_1$ and $Q \in L_1$.
	Without loss of generality, we can assume that $P$ and $Q$ lie in $I_1$. Then, this means that $P$ lies in $x_n = -\beta_2$ and $Q$ lies in $x_n = -\beta_2$, and thus, we obtain
	\begin{align*}
		P (d_{l 1}, \cdots, d_{l (n - 1)}, -\beta_2), \ 
		Q (e_{v 1}, \cdots, e_{v (n - 1)}, -\beta_2).
	\end{align*}
	
	Next, for segments $NP$ and $NQ$, we have
	\begin{align*}
		|NP| = \sqrt{ (d_{l 1})^2 +  \cdots + (d_{l (n - 1)})^2 + (-\beta_2 + \beta_2 )^2}, \\
		|NQ| = \sqrt{ (e_{v 1})^2 + \cdots + (e_{v (n - 1)})^2 + (-\beta_2 + \beta_2)^2},
	\end{align*}
	where $|\cdot|$ denotes the length of segment.
	
	It follows from (\ref{eq13}) that $|NP| \geq |NQ|$, and therefore, for $P \in K_1$ and $Q \in L_1$, $|NP| \geq |NQ|$ holds in $I_1, \cdots, I_{2^{n - 1}}$. 
	Similarly, in $I_{2^{n - 1} + 1}$, we get that $P$ lies in $x_n = -\beta_1$ and $Q$ lies in $x_n = -\beta_1$, and then, by (\ref{eq12}), it is easy to verify that $|MP| \geq |MQ|$. 
	Hence, in $I_{2^{n - 1} + 1}, \cdots, I_{2^n}$, we can get that, for $P \in K_1$ and $Q \in L_1$, $|MP| \geq |MQ|$.
	
	See Fig. \ref{Fig:figure13}, which shows the situation of $K \subset L$ for $n = 3$. 
	\begin{figure}[h]
		\centering
		\subfigure[$K \subset L $ ]{
			\label{Fig:figure9}
			\includegraphics[width=0.26\textwidth]{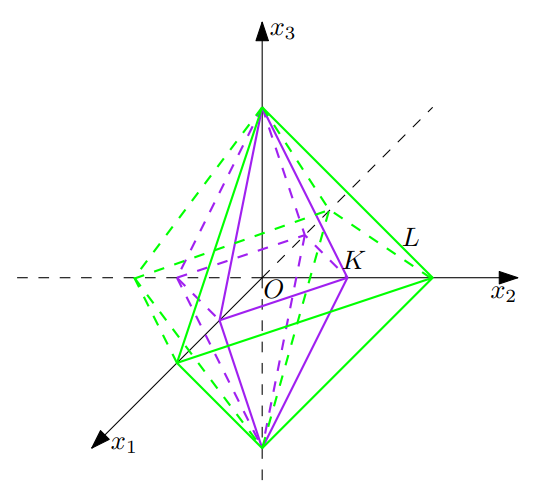}}
		\subfigure[$L_{\mathbb{Z}^3}^* \subset K_{\mathbb{Z}^3}^*$]{
			\label{Fig:figure11}
			\includegraphics[width=0.26\textwidth]{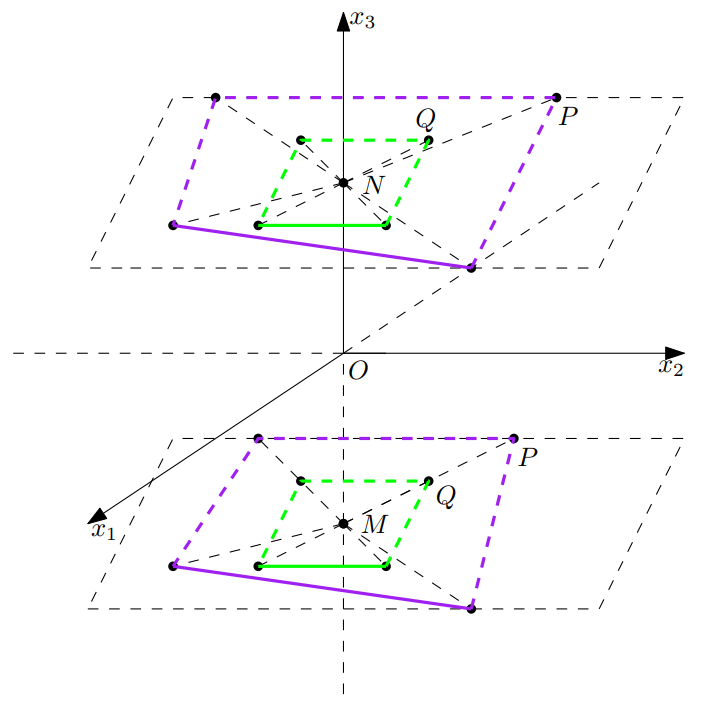}}
		\caption{An illustration of the inclusion relation}
		\label{Fig:figure13}
	\end{figure}
	
	Summarizing the above arguments, we obtain Table \ref{Fig:figure10}.
	\begin{table*}[h]
		\centering        
		\caption{\label{Fig:figure10} Results of the segment comparison in $x_n = -\beta_1$ and $x_n = -\beta_2$}
		\begin{tabular}{cccc }
			\hline\hline\noalign{\smallskip}	
			Blocks               &Characteristics of coordinates   & $x_n = -\beta_1$  & $x_n = -\beta_2$ \\
			\noalign{\smallskip}\hline\noalign{\smallskip}
			$I_1$	            &$(+, +, +,\cdots, +, +)$     &                   & $|NP| \geq |NQ|$ \\
			$I_2$	            &$(-, +, +, \cdots, +, +)$    &                   & $|NP| \geq |NQ|$ \\
			$I_3$	            &$(-, +, -, \cdots, +, +)$    &                   & $|NP| \geq |NQ|$ \\
			$\vdots$            &$\vdots$                     &                   & $\vdots$         \\
			$I_{2^{n - 2} + 1}$ &$(-, -, -, \cdots, -, +)$    &                   & $|NP| \geq |NQ|$ \\
			$I_{2^{n - 2} + 2}$ &$(+, -, +, \cdots, +, +)$    &                   & $|NP| \geq |NQ|$  \\
			$I_{2^{n - 2} + 3}$ &$(+, +, -, \cdots, +, +)$    &                   & $|NP| \geq |NQ|$  \\
			$\vdots$            &$\vdots$                     &                   & $\vdots$          \\
			$I_{2^{n - 1}}$     &$(+, -, -, \cdots, -, +)$    &                   & $|NP| \geq |NQ|$ \\
			$I_{2^{n - 1} + 1}$ &$(+, +, +, \cdots, +, -)$    & $|MP| \geq |MQ|$  &                  \\
			$I_{2^{n - 1} + 2}$ &$(-, +, +, \cdots, +, -)$    & $|MP| \geq |MQ|$  &                  \\
			$I_{2^{n - 1} + 3}$ &$(-, +, -, \cdots, +, -)$    & $|MP| \geq |MQ|$  &                  \\
			$\vdots$            &$\vdots$                     & $\vdots$          &                  \\
			$I_{3 \cdot 2^{n - 2} + 1}$ &$(-, -, -, \cdots, -, -)$    & $|MP| \geq |MQ|$  &  \\
			$I_{3 \cdot 2^{n - 2} + 2}$ &$(+, -, +, \cdots, +, -)$    & $|MP| \geq |MQ|$  &   \\
			$I_{3 \cdot 2^{n - 2} + 3}$ &$(+, +, -, \cdots, +, -)$    & $|MP| \geq |MQ|$  &   \\
			$\vdots$                    &$\vdots$                     & $\vdots$          &          \\
			$I_{2^n}$                   &$(+, -, -, \cdots, -, -)$    & $|MP| \geq |MQ|$  &       \\
			\noalign{\smallskip}\hline
		\end{tabular}
	\end{table*}  
	
	Assume that $P_{\beta_1 i}, P_{\beta_2 l} \in K_1$ and $Q_{\beta_1 u}, Q_{\beta_2 v} \in L_1$, and then, we let
	\begin{equation}\label{pdf4}
		\begin{gathered}
			P_{\beta_1 i}  = (b_{i 1}, \cdots, b_{i (n - 1)}, -\beta_1 ),\
			P_{\beta_2 l} = (d_{l 1}, \cdots, d_{l (n - 1)}, -\beta_2 ), \\
			Q_{\beta_1 u}  = (c_{u 1}, \cdots, c_{u (n - 1)}, -\beta_1 ), \
			Q_{\beta_2 v} = (e_{v 1}, \cdots, e_{v (n - 1)}, -\beta_2 ),
		\end{gathered}
	\end{equation}
	for $1 \leq i \leq m_1$, $1 \leq l \leq m_2$, $1 \leq u \leq s_1$, $1 \leq v \leq s_2$.
	
	Indeed, if $K \subseteq L$, then $L_{\mathbb{Z}^n}^* \subseteq K_{\mathbb{Z}^n}^*$, which is equivalent to
	\begin{align*}		
		\mathrm{conv}(\bigcup_{u = 1}^{s_1} Q_{\beta_1 u})  \bigcap \mathbb{Z}^n \subseteq \mathrm{conv}(\bigcup_{i = 1}^{m_1} P_{\beta_1 i})  \bigcap \mathbb{Z}^n
	\end{align*}  
	in $x_n = -\beta_1$ 
	and 
	\begin{align*}	
		\mathrm{conv}(\bigcup_{v = 1}^{s_2} Q_{\beta_2 v})  \bigcap \mathbb{Z}^n \subseteq \mathrm{conv}(\bigcup_{l = 1}^{m_2} P_{\beta_2 l})  \bigcap \mathbb{Z}^n 
	\end{align*} 
	in $x_n = -\beta_2$.
	
	For convenience, we let 
	\begin{align*}		
		U_{P_{\beta_1}} = \bigcup_{i = 1}^{m_1}P_{\beta_1 i},\
		U_{P_{\beta_2}} = \bigcup_{l = 1}^{m_2}P_{\beta_2 l},\
		U_{Q_{\beta_1}} = \bigcup_{u = 1}^{s_1}Q_{\beta_1 u},\
		U_{Q_{\beta_2}} = \bigcup_{v = 1}^{s_2}Q_{\beta_2 v}.
	\end{align*} 
	
	Obviously, we have $|NP| \geq |NQ|$ in $\{x_n = -\beta_2\} \cap I_{t}$ for $1 \leq t \leq 2^{n - 1}$, see Table \ref{Fig:figure10}. Then, we obtain
	\begin{align*}
		U_{Q_{\beta_2}} \subseteq U_{P_{\beta_2}}.
	\end{align*}
	It follows from Lemma \ref{q1} that
	\begin{align}\label{c1}
		\mathrm{conv}(U_{Q_{\beta_2}}) \cap \mathbb{Z}^n \subseteq \mathrm{conv}(U_{P_{\beta_2}}) \cap \mathbb{Z}^n
	\end{align}
	in $x_n = -\beta_2$. Analogously, we have $|MP| \geq |MQ|$ in $\{x_n = -\beta_1\} \cap I_{2^{n - 1} + t}$ for $1 \leq t \leq 2^{n - 1}$, see Table \ref{Fig:figure10}. Thus, we see that
	\begin{align*}
		U_{Q_{\beta_1}} \subseteq U_{P_{\beta_1}},
	\end{align*}
	and then, by Lemma \ref{q1}, in $x_n = -\beta_1$, we get that
	\begin{align}\label{c2}
		\mathrm{conv}(U_{Q_{\beta_1}}) \cap \mathbb{Z}^n \subseteq \mathrm{conv}(U_{P_{\beta_1}}) \cap \mathbb{Z}^n.
	\end{align} 
	
	Together with (\ref{c1}) and (\ref{c2}), we have
	\begin{align*}
		(\mathrm{conv}(U_{Q_{\beta_2}}) \cap \mathbb{Z}^n ) \cup ( \mathrm{conv}(U_{Q_{\beta_1}}) \cap \mathbb{Z}^n ) 
		\subseteq  
		(\mathrm{conv}(U_{P_{\beta_2}}) \cap \mathbb{Z}^n ) \cup (\mathrm{conv}(U_{P_{\beta_1}}) \cap \mathbb{Z}^n),
	\end{align*}
	and then, we get that
	\begin{align*}
		& \quad \ \mathrm{conv} \{ (\mathrm{conv}(U_{Q_{\beta_2}}) \cap \mathbb{Z}^n ) \cup ( \mathrm{conv}(U_{Q_{\beta_1}}) \cap \mathbb{Z}^n ) \}\cap \mathbb{Z}^n \\
		& \subseteq  
		\mathrm{conv} \{ (\mathrm{conv}(U_{P_{\beta_2}}) \cap \mathbb{Z}^n ) \cup (\mathrm{conv}(U_{P_{\beta_1}}) \cap \mathbb{Z}^n) \} \cap \mathbb{Z}^n.
	\end{align*}
	
	It follows from Lemma \ref{b2} that
	\begin{align*}
		\mathrm{conv}\{ (U_{Q_{\beta_1}} \cup U_{Q_{\beta_2}})\} \cap \mathbb{Z}^n \subseteq \mathrm{conv}\{(U_{P_{\beta_1}} \cup U_{P_{\beta_2}})\} \cap \mathbb{Z}^n.
	\end{align*}
	
	This yields that
	$$L_{\mathbb{Z}^n}^* \subseteq K_{\mathbb{Z}^n}^*.$$
\end{proof}

We illustrate Theorem \ref{key4} with an example in $\mathbb{Z}^2$.
\begin{example}
	Let $K, L \in \mathscr{C}_{\mathbb{Z}}^2$.
	The vertices of $K$ and $L$ are 
	$K^v = \{(0,3), (1,0), (0,-3),  (-1,0)\}$ and 
	$L^v = \{(0,3), (2,1), $ $(0,-3), (-2,1)\}$, respectively.
	Then, we can get that 
	$(K_{\mathbb{Z}^2}^*)^v = \{(-3,-3), (3,3), (-3,3), (3,-3)\}$ 
	and 
	$ (L_{\mathbb{Z}^2}^*)^v = \{(-1,-3),$ $ (2,3), (-2,3), (1,-3)\}$.
	Therefore, we have $L_{\mathbb{Z}^2}^* \subset K_{\mathbb{Z}^2}^*$.
	See Fig. \ref{Fig:figure3}.
	
	\begin{figure}[H]                       
		\includegraphics[width=0.50\textwidth]{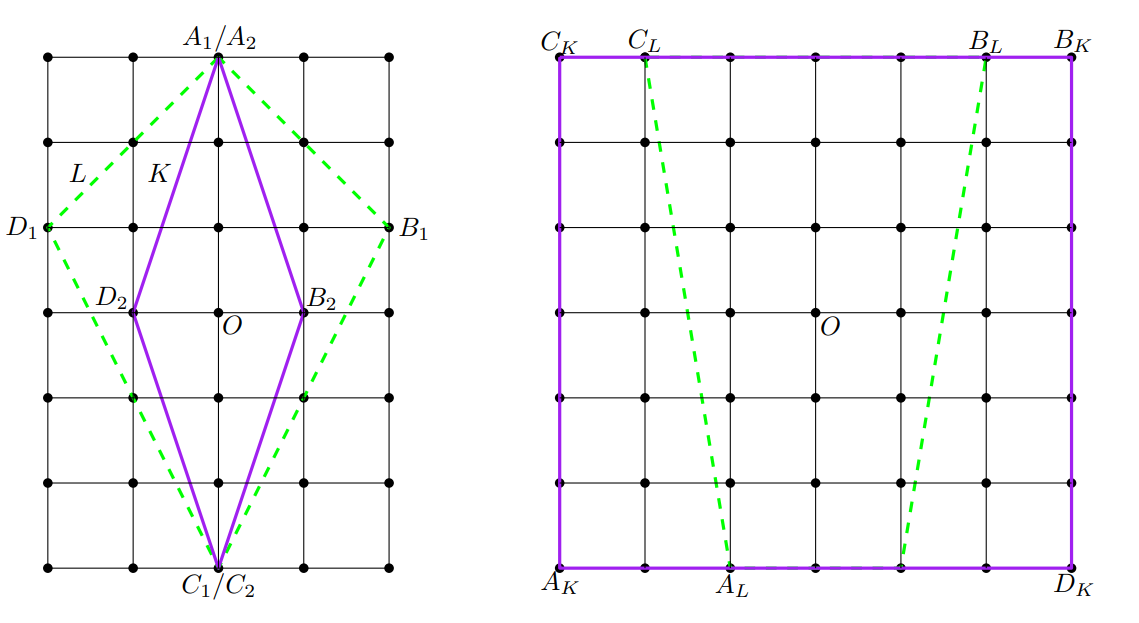}
		\caption{\label{Fig:figure3}$K$ (purple) and $L$ (green), \ $K_{\mathbb{Z}^2}^*$ and $ L_{\mathbb{Z}^2}^*$ (from left)}
	\end{figure}
\end{example}

We now come to intersection and union on the polar of convex lattice sets. 

\begin{theorem}\label{pdf1}
	If $ K,L \in \mathscr{C}_{\mathbb{Z}}^n$, then 
	$(K \cap L)_{\mathbb{Z}^n}^* = \mathrm{conv}(K_{\mathbb{Z}^n}^* \cup L_{\mathbb{Z}^n}^*) \cap \mathbb{Z}^n.$
\end{theorem}

\begin{proof}
	For convenience, we let
	\begin{gather*}
		J_{h_{\beta_1}} = \bigcap_{i = 1}^{m_1} h_{\beta_1 i}^-,\
		\ J_{h_{\beta_2}} = \bigcap_{l = 1}^{m_2} h_{\beta_2 l}^-,\ 
		\ J_{g_{\beta_1}} = \bigcap_{u = 1}^{s_1} g_{\beta_1 u}^-,\
		\ \ J_{g_{\beta_2}} = \bigcap_{v = 1}^{s_2} g_{\beta_2 v}^-, \\		
		U_{P_{\beta_1}} = \bigcup_{i = 1}^{m_1}P_{\beta_1 i},\
		U_{P_{\beta_2}} = \bigcup_{l = 1}^{m_2}P_{\beta_2 l},\
		U_{Q_{\beta_1}} = \bigcup_{u = 1}^{s_1}Q_{\beta_1 u},\
		U_{Q_{\beta_2}} = \bigcup_{v = 1}^{s_2}Q_{\beta_2 v}.
	\end{gather*} 
	
	Clearly, $K$ and $L$ are bounded by the hyperplanes of (\ref{eq4}) and (\ref{eq5}), respectively. Then, we have 
	\begin{align}\label{K}
		K = \{ ( J_{h_{\beta_1}} \cap \{x_n \geq 0\}  ) \cup ( J_{h_{\beta_2}} \cap \{x_n < 0\} )  \} \cap \mathbb{Z}^n,
	\end{align}
	\begin{align}\label{L}
		L = \{ ( J_{g_{\beta_1}} \cap \{x_n \geq 0\}  ) \cup ( J_{g_{\beta_2}} \cap \{x_n < 0\} )  \} \cap \mathbb{Z}^n.
	\end{align}
	
	By the definitions of $K$ and $L$, we have $K \cap L \neq \varnothing$.
	Together with (\ref{K}) and (\ref{L}), we get that
	\begin{align*}
		K \cap L 
		& = \{ \{ ( J_{h_{\beta_1}} \cap \{x_n \geq 0\} ) \cup ( J_{h_{\beta_2}} \cap  \{x_n < 0\} )   \} \cap \mathbb{Z}^n  \} 
		\cap \\
		& \quad \  
		\{  \{ ( J_{g_{\beta_1}} \cap  \{x_n \geq 0\}  ) \cup ( J_{g_{\beta_2}} \cap   \{x_n < 0\} ) \}  \cap \mathbb{Z}^n  \}, 
	\end{align*}
	from which it follows that
	\begin{align*}
		K \cap L 
		& = \{ \{ ( J_{h_{\beta_1}} \cap \{x_n \geq 0\}   ) \cup ( J_{h_{\beta_2}} \cap  \{x_n < 0\} )  \} 
		\cap 
		\{ ( J_{g_{\beta_1}} \cap \{x_n \geq 0\} ) \cup ( J_{g_{\beta_2}} \cap \{x_n < 0\} )  \}  \} \cap \mathbb{Z}^n \\
		& =\{ \{ ( J_{h_{\beta_1}} \cap  \{x_n \geq 0\} ) \cap ( J_{g_{\beta_1}} \cap \{x_n \geq 0\} ) \}  \cup 
		\{( J_{h_{\beta_2}} \cap 
		\{x_n < 0\}  ) \cap 
		( J_{g_{\beta_1}} \cap \{x_n \geq 0\} )  \} 
		\cup \\
		& \quad \
		\{( J_{h_{\beta_1}} \cap \{x_n \geq 0\}) \cap ( J_{g_{\beta_2}} \cap \{x_n < 0\}  ) \}  
		\cup 
		\{ ( J_{h_{\beta_2}} \cap \{x_n < 0\}   ) \cap ( J_{g_{\beta_2}} \cap \{x_n < 0\}  )\} \} \cap \mathbb{Z}^n.
	\end{align*}
	
	Obviously, we have
	\begin{align*}		
		\{ ( J_{h_{\beta_2}} \cap \{x_n < 0\}  ) 
		\cap  ( J_{g_{\beta_1}} \cap \{x_n \geq 0\} )  \} = \varnothing, \\
		\{ ( J_{h_{\beta_1}} \cap \{x_n \geq 0\} ) \cap ( J_{g_{\beta_2}} \cap \{x_n < 0\} ) \}
		= \varnothing,
	\end{align*}
	which means that  
	\begin{align*}		
		K \cap L 
		& =  \{ \{ ( J_{h_{\beta_1}} \cap \{x_n \geq 0\} ) \cap ( J_{g_{\beta_1}} \cap \{x_n \geq 0\}  ) \} 
		\cup \\
		& \quad \
		\{ ( J_{h_{\beta_2}} \cap \{x_n < 0\}  ) \cap ( J_{g_{\beta_2}} \cap \{x_n < 0\}  ) \}  \} \cap \mathbb{Z}^n.
	\end{align*}
	
	For $x_n = 0$, the discrete Legendre transform $x_n^*(p) : \mathbb{Z}^{n - 1} \rightarrow \overline{\mathbb{Z}}$ is given by
	$$ 
	x_n^*(p) =\left\{
	\begin{aligned}
		0,        \ \quad & p=(0, \cdots, 0), \\
		+ \infty, \ \quad & {\rm{otherwise.}}
	\end{aligned}
	\right.
	$$
	Obviously, the origin $O \in (K \cap L)_{\mathbb{Z}^n}^* $. 
	From (\ref{pdf4}), Definition \ref{key1} and Definition \ref{key2}, we get that 
	\begin{align*}
		(K \cap L)_{\mathbb{Z}^n}^* 
		=  \mathrm{conv}\{( U_{P_{\beta_1}} \cup U_{Q_{\beta_1}}) \cup  
		(U_{P_{\beta_2}} \cup U_{Q_{\beta_2}})\} \cap \mathbb{Z}^n.
	\end{align*}
	
	Now we show that $$\mathrm{conv} (K_{\mathbb{Z}^n}^* \cup L_{\mathbb{Z}^n}^* ) \cap \mathbb{Z}^n = \mathrm{conv}\{ (U_{P_{\beta_1}} \cup U_{Q_{\beta_1}} ) \cup (U_{P_{\beta_2}} \cup U_{Q_{\beta_2}})\} \cap \mathbb{Z}^n.$$
	
	By (\ref{pdf4}), (\ref{conv1}) and (\ref{conv2}), we can see that  
	\begin{align}\label{pdf2}
		K_{\mathbb{Z}^n}^* = \mathrm{conv} (U_{P_{\beta_1}} \cup U_{P_{\beta_2}}) \cap \mathbb{Z}^n,
	\end{align}
	\begin{align}\label{pdf3}
		L_{\mathbb{Z}^n}^* = \mathrm{conv} (U_{Q_{\beta_1}} \cup U_{Q_{\beta_2}}) \cap \mathbb{Z}^n.
	\end{align} 
	Consequently, we have 
	\begin{align*}		
		\mathrm{conv} (K_{\mathbb{Z}^n}^* \cup L_{\mathbb{Z}^n}^* ) \cap \mathbb{Z}^n
		= \mathrm{conv}\{ ( \mathrm{conv}(U_{P_{\beta_1}} \cup U_{P_{\beta_2}}) \cap \mathbb{Z}^n ) \cup  ( \mathrm{conv}(U_{Q_{\beta_1}} \cup U_{Q_{\beta_2}}) \cap \mathbb{Z}^n  )  \} \cap \mathbb{Z}^n.
	\end{align*}
	
	It follows from Lemma \ref{b2} that 
	\begin{align*}
		\mathrm{conv} (K_{\mathbb{Z}^n}^* \cup L_{\mathbb{Z}^n}^* ) \cap \mathbb{Z}^n
		& = \mathrm{conv} \{(U_{P_{\beta_1}} \cup U_{P_{\beta_2}}) \cup (U_{Q_{\beta_1}} \cup U_{Q_{\beta_2}}) \} \cap \mathbb{Z}^n \\
		& = \mathrm{conv}\{ (U_{P_{\beta_1}} \cup U_{Q_{\beta_1}} ) \cup (U_{P_{\beta_2}} \cup U_{Q_{\beta_2}})\} \cap \mathbb{Z}^n.
	\end{align*}
	
	Thus, we have 
	$$(K \cap L)_{\mathbb{Z}^n}^* = \mathrm{conv} (K_{\mathbb{Z}^n}^* \cup L_{\mathbb{Z}^n}^* ) \cap \mathbb{Z}^n .$$
\end{proof}

We illustrate Theorem \ref{pdf1} with an example in $\mathbb{Z}^2$.

\begin{example}
	Let $K, L \in \mathscr{C}_{\mathbb{Z}}^2$.
	The vertices of $K$ is 
	$K^v = \{(0,3), (-1,0), (0,-3), $ $(2,-1)\}$
	and the vertices of $L$ is 
	$L^v = \{(0,3), (-3,0)$, $(0,-3), (1,0)\}.$
	Then, we get that 
	$(K_{\mathbb{Z}^2}^*)^v = \{(3,-3), (-3,3), (1,3), (-2,-3)\}$
	and
	$(L_{\mathbb{Z}^2}^*)^v = \{(1,-3), (-1,3), $ $(3,3), (-3,-3)\}$.
	Thus,
	$(K \cap L)_{\mathbb{Z}^2}^* = \mathrm{conv}(K_{\mathbb{Z}^2}^* \cup L_{\mathbb{Z}^2}^* ) \cap \mathbb{Z}^2$.
	See Fig. \ref{Fig:figure4}.
	
	\begin{figure}[H]                       
		\includegraphics[width=0.58\textwidth]{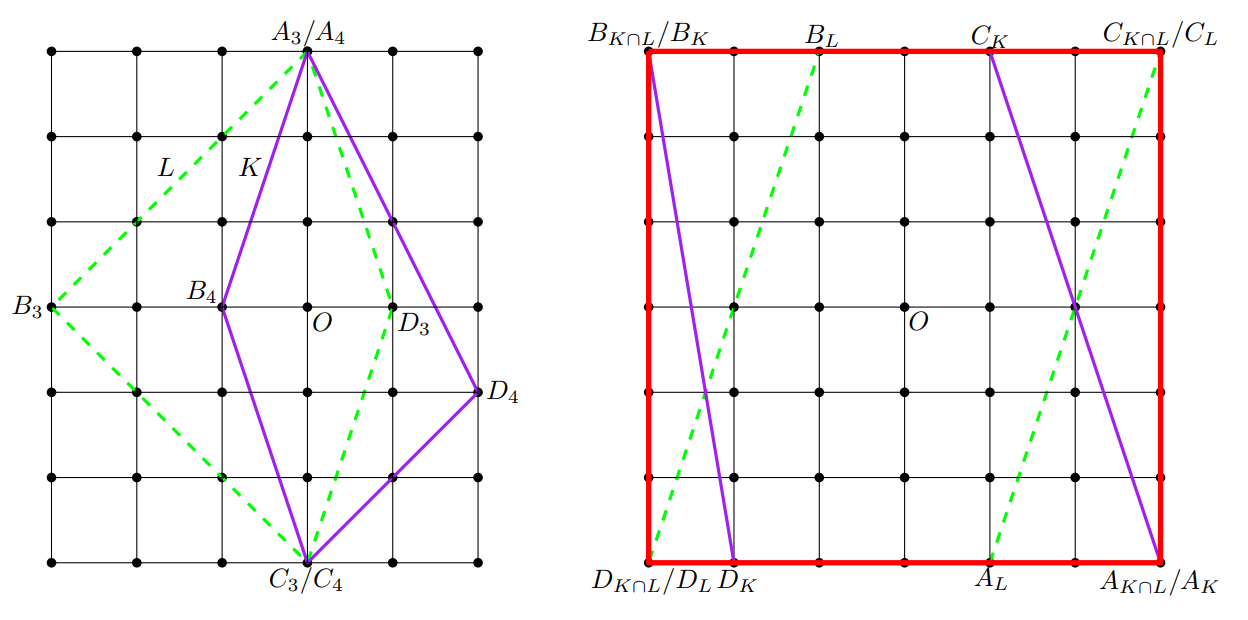}
		\caption{\label{Fig:figure4}$K$ (purple) and $L$ (green), \  $ \mathrm{conv} (K_{\mathbb{Z}^2}^* \cup  L_{\mathbb{Z}^2}^* ) \cap \mathbb{Z}^2$ and $(K \cap L)_{\mathbb{Z}^2}^*$ (from left) }
	\end{figure}
\end{example}

From now on, we consider the intersection of the polar of convex lattice sets.

\begin{theorem}\label{key5}
	If $K,L \in \mathscr{C}_{\mathbb{Z}}^n$, $K \subseteq L$ or $L \subseteq K$, then 
	$$(K \cup L)_{\mathbb{Z}^n}^* = K_{\mathbb{Z}^n}^* \cap L_{\mathbb{Z}^n}^*.$$
\end{theorem}

\begin{proof}
	Without loss of generality, we assume that $K \subseteq L$. By Theorem \ref{key4}, $K \subseteq L$ implies $L_{\mathbb{Z}^n}^* \subseteq K_{\mathbb{Z}^n}^*$. Obviously, we get that $$(K \cup L)_{\mathbb{Z}^n}^* = K_{\mathbb{Z}^n}^* \cap L_{\mathbb{Z}^n}^*.$$
\end{proof}

For $K, L \in \mathscr{C}_{\mathbb{Z}}^n$, 
we assume that $K$ is generated by the intersection of hyperplanes as follows: 
\begin{equation}\label{a11}
	\begin{aligned}
		h_{\beta_1 i}(x) & = b_{i 1}x_1 + \cdots + b_{i (n-1)}x_{n-1} + \beta_1, \ {\rm{for}} \ \beta_1 > 0, i = 1, \cdots, m_1, \\
		h_{\beta_2 l}(x) & = d_{l 1}x_1 + \cdots + d_{l (n-1)}x_{n-1} + \beta_2, \ {\rm{for}} \ \beta_2 < 0, l = 1, \cdots, m_2,
	\end{aligned}
\end{equation}
with $b_{i j}, d_{l j}, \beta_1, \beta_2 \in \mathbb{Z}$ for $j = 1, \cdots, n - 1$.

Analgously, $L$ is bounded by the following hyperplanes:
\begin{equation}\label{a12}
	\begin{aligned}
		g_{\beta_1 u}(x) & = c_{u 1}x_1 + \cdots + c_{u (n - 1)}x_{n - 1} + \beta_1, \ {\rm{for}} \ \beta_1 > 0, u = 1, \cdots, s_1,\\
		g_{\beta_2 v}(x) & = e_{v 1}x_1 + \cdots + e_{v (n - 1)}x_{n - 1} + \beta_2, \ {\rm{for}} \ \beta_2 < 0,v = 1, \cdots, s_2,
	\end{aligned}
\end{equation}
with $c_{u j}, e_{v j}, \beta_1, \beta_2 \in \mathbb{Z}$ for $j = 1, \cdots, n - 1$.

In the following, we shall now give two special convex lattice sets $K$ and $L$ satisfying the following conditions:
\begin{enumerate}
	\item \label{k1} $K$ and $L$ are bounded by the hyperplanes (\ref{a11}) and (\ref{a12}), respectively.
	\item \label{k2} The coefficients of hyperplanes enclosing $K$ and $L$ satisfying $- b_{j 1} > b_{j 2}$ or $b_{j 2} > - b_{j 1}$ with $b_{j 2} \leq -1, b_{j 1} \geq 1$ for $j = 1, \cdots, n - 1$.
\end{enumerate} 

Conditions \ref{k1} and \ref{k2} are equivalent to Table \ref{Fig:figure16}.

\begin{sidewaystable}
	\caption{The coefficients of hyperplanes enclosing $K$ and $L$}\label{Fig:figure16}
	\begin{threeparttable}
		\begin{tabular}{ccccccccccccccc}
			\hline\hline\noalign{\smallskip}	
			&	&  & \multicolumn{5}{@{}c@{}}{$K$}&   & \multicolumn{5}{@{}c@{}}{$L$} & \\ \cmidrule{4-8}\cmidrule{10-14}%
			Blocks	&Characteristics of coordinates  &    & $x_1$    & $x_2$  &$x_3$  & $\cdots$   & $x_{n - 1}$ & & $x_1$ & $x_2$ &$x_3$ & $\cdots$ & $x_{n - 1}$ & $x_n$   \\
			\noalign{\smallskip}\hline\noalign{\smallskip}
			$I_1 \cap \{x_n \leq \beta_1\}$   &	$(+, +, +, \cdots, +, +)$  &  & $b_{12}$    & $b_{22}$ & $b_{32}$   & $\cdots$ & $b_{(n-1)2}$ & & $-b_{11}$ & $-b_{21}$ & $-b_{31}$  & $\cdots$ & $-b_{(n-1)1}$  & $\beta_1$ \\
			$I_2 \cap \{x_n \leq \beta_1\}$	&$(-, +, +, \cdots, +, +)$  &  & $b_{11}$    & $b_{22}$ & $b_{32}$ & $\cdots$ & $b_{(n-1)2}$ & &$-b_{12}$ & $-b_{21}$ & $-b_{31}$ & $\cdots$ & $-b_{(n-1)1}$& $\beta_1$ \\
			$I_3 \cap \{x_n \leq \beta_1\}$	&$(-, +, -, \cdots, +, +)$  &  & $b_{11}$    & $b_{22}$ & $b_{31}$ & $\cdots$ & $b_{(n-1)2}$ & &$-b_{12}$ & $-b_{21}$ & $-b_{32}$ & $\cdots$ & $-b_{(n-1)1}$& $\beta_1$ \\
			$\vdots$&$\vdots$             &     & $\vdots$    & $\vdots$ & $\vdots$ & $\vdots$  & $\vdots$ &  &$\vdots$ &$\vdots$  & $\vdots$ & $\vdots$& $\vdots$ & $\vdots$  \\
			$I_{2^{n - 2} + 1} \cap \{x_n \leq \beta_1\}$   &	$(-, -, -, \cdots, -, +)$  &  & $b_{11}$    & $b_{21}$ & $b_{31}$  & $\cdots$ & $b_{(n-1)1}$ & & $-b_{12}$ & $-b_{22}$ & $-b_{32}$ & $\cdots$ & $-b_{(n-1)2}$  & $\beta_1$ \\
			$I_{2^{n - 2} + 2} \cap \{x_n \leq \beta_1\}$   &	$(+, -, +, \cdots, +, +)$  &  & $b_{12}$    & $b_{21}$ & $b_{32}$ & $\cdots$ & $b_{(n-1)2}$ & & $-b_{11}$ & $-b_{22}$ & $-b_{31}$& $\cdots$ & $-b_{(n-1)1}$  & $\beta_1$ \\
			$I_{2^{n - 2} + 3} \cap \{x_n \leq \beta_1\}$   &	$(+, +, -, \cdots, +, +)$  &  & $b_{12}$    & $b_{22}$ & $b_{31}$ & $\cdots$ & $b_{(n-1)2}$ & & $-b_{11}$ & $-b_{21}$ & $-b_{32}$& $\cdots$ & $-b_{(n-1)1}$  & $\beta_1$ \\
			$\vdots$&$\vdots$             &     & $\vdots$    & $\vdots$ & $\vdots$  & $\vdots$ & $\vdots$ & &$\vdots$ &$\vdots$ &$\vdots$  & $\vdots$ & $\vdots$ & $\vdots$  \\
			$I_{2^{n - 1}} \cap \{x_n \leq \beta_1\}$	&$(+, -, -,\cdots, -, +)$  &  & $b_{12}$    & $b_{21}$ & $b_{31}$  & $\cdots$ & $b_{(n-1)1}$ &  &$-b_{11}$ &$-b_{22}$ & $-b_{32}$ & $\cdots$ & $-b_{(n-1)2}$ & $\beta_1$ \\
			$I_{2^{n - 1} + 1} \cap \{x_n \geq \beta_2\}$ &$(+, +,+, \cdots, +, -)$ &   & $b_{11}$    & $b_{21}$ & $b_{31}$ & $\cdots$ & $b_{(n-1)1}$&  & $-b_{12}$ &$-b_{22}$ & $-b_{32}$& $\cdots$ & $-b_{(n-1)2}$& $\beta_2$ \\
			$I_{2^{n - 1} + 2} \cap \{x_n \geq \beta_2\}$ &$(-, +, +, \cdots, +, -)$ &   & $b_{12}$    & $b_{21}$ & $b_{31}$ & $\cdots$ & $b_{(n-1)1}$&  & $-b_{11}$ &$-b_{22}$ & $-b_{32}$& $\cdots$ & $-b_{(n-1)2}$& $\beta_2$ \\
			$I_{2^{n - 1} + 3} \cap \{x_n \geq \beta_2\}$ &$(-, +, -, \cdots, +, -)$ &   & $b_{12}$    & $b_{21}$ & $b_{32}$ & $\cdots$ & $b_{(n-1)1}$&  & $-b_{11}$ &$-b_{22}$ & $-b_{31}$& $\cdots$ & $-b_{(n-1)2}$& $\beta_2$ \\
			$\vdots$&$\vdots$             &     & $\vdots$ & $\vdots$   & $\vdots$   & $\vdots$ & $\vdots$ & &$\vdots$ &$\vdots$&$\vdots$  & $\vdots$ & $\vdots$ & $\vdots$  \\
			$I_{3 \cdot 2^{n - 2} + 1} \cap \{x_n \geq \beta_2\}$ &$(-, - ,-, \cdots, -, -)$ &   & $b_{12}$    & $b_{22}$  & $b_{32}$& $\cdots$ & $b_{(n-1)2}$&  & $-b_{11}$ &$-b_{21}$ & $-b_{31}$& $\cdots$ & $-b_{(n-1)1}$& $\beta_2$ \\
			$I_{3 \cdot 2^{n - 2} + 2} \cap \{x_n \geq \beta_2\}$ &$(+, -, +,\cdots, +, -)$ &   & $b_{11}$    & $b_{22}$ & $b_{31}$ & $\cdots$ & $b_{(n-1)1}$&  & $-b_{12}$ &$-b_{21}$ & $-b_{32}$& $\cdots$ & $-b_{(n-1)2}$& $\beta_2$ \\
			$I_{3 \cdot 2^{n - 2} + 3} \cap \{x_n \geq \beta_2\}$ &$(+, +,-, \cdots, +, -)$ &   & $b_{11}$    & $b_{21}$ & $b_{32}$ & $\cdots$ & $b_{(n-1)1}$&  & $-b_{12}$ &$-b_{22}$ & $-b_{31}$& $\cdots$ & $-b_{(n-1)2}$& $\beta_2$ \\
			$\vdots$ &$\vdots$           &       & $\vdots$ & $\vdots$   & $\vdots$ & $\vdots$  & $\vdots$  &   & $\vdots$  & $\vdots$& $\vdots$ & $\vdots$& $\vdots$ & $\vdots$ \\
			$I_{2^n} \cap \{x_n \geq \beta_2\}$ &$(+,  -, -,\cdots, -, -)$ &   & $b_{11}$    & $b_{22}$ & $b_{32}$   & $\cdots$ & $b_{(n-1)2}$ & &$-b_{12}$ &$-b_{21}$ & $-b_{31}$& $\cdots$  & $-b_{(n-1)1}$ & $\beta_2$ \\
			\noalign{\smallskip}\hline
		\end{tabular}
		\begin{tablenotes}    
			\footnotesize               
			\item[1] From Table \ref{Fig:figure16}, for any two blocks with opposite coordinate signs, it is easy to get that the hyperplanes enclosing $K$ are parallel to each other. In other words, Table \ref{Fig:figure16} gives a one-to-one correspondence between the any two blocks with opposite coordinate signs and the coefficients of hyperplanes. 
			For example, in $I_1 \cap \{x_n \leq \beta_1\}$ and $I_{3 \cdot 2^{n - 2} + 1} \cap \{x_n \geq \beta_2\}$, the coefficients of hyperplanes are equal, and then, it is easy to get that the two hyperplanes are parallel.
			Obviously,  
			$K$ is centrally symmetric. 
			Analogously, by Table \ref{Fig:figure16}, 
			$L$ is also centrally symmetric. 
			Clearly, by the definitions of $K$ and $L$, $K$ and $L$ have the same center of symmetry (possibly not the origin).
			\item[2] According to Table \ref{Fig:figure16}, we know that $m_1 = m_2 = s_1 = s_2 = 2^{n - 1}$, $b_{j 1} \geq 1, b_{j 2} \leq -1$ and $b_{j 1}, b_{j 2} \in \mathbb{Z}$ for $ j = 1, \cdots, n - 1$.
		\end{tablenotes}     
	\end{threeparttable}
\end{sidewaystable}

Fig. \ref{Fig:figure14} shows, as two examples, the coefficients of hyperplanes enclosing $K$ and $L$ satisfying $- b_{j 1} > b_{j 2}$ or $b_{j 2} > - b_{j 1}$, respectively, for $n = 2$. 
\begin{figure}[h]
	\centering
	\subfigure[$b_{j 2} > - b_{j 1}$]{
		\label{Fig:figure17}
		\includegraphics[width=0.25\textwidth]{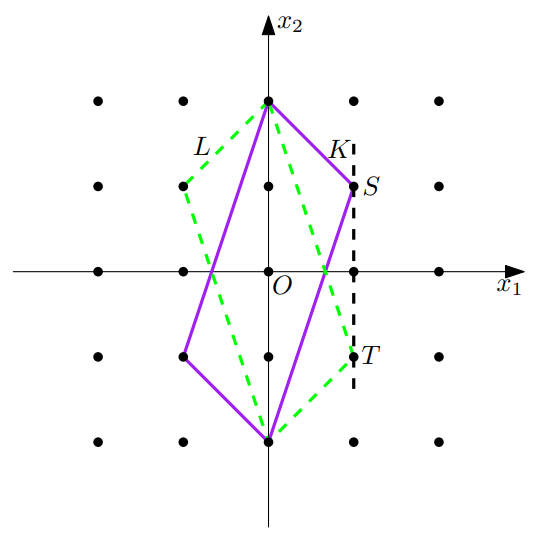}}
	\subfigure[$- b_{j 1} > b_{j 2}$]{
		\label{Fig:figure18}
		\includegraphics[width=0.25\textwidth]{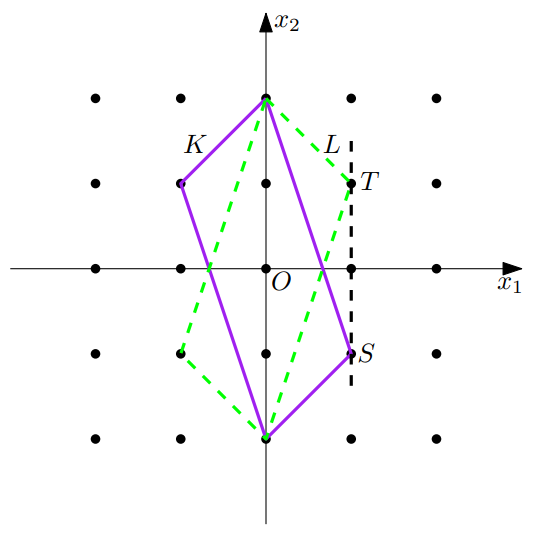}}
	\caption{$K$ (purple) and $L$ (green) in $\mathbb{Z}^2$ }
	\label{Fig:figure14}
\end{figure}

\begin{remark}
	Following Table \ref{Fig:figure16}, in $x_n \geq 0$, for any two blocks whose coordinate signs are opposite except for the $n$-th coordinate, we get that the coefficients of hyperplanes enclosing $K$ and $L$ are opposite to each other, and then, we see that the hyperplanes bounded by $K$ and $L$ are symmetric with respect to the $x_n$-axis, see, e.g., $I_1 \cap \{x_n \leq \beta_1\}$ and $I_{2^{n - 2} + 1} \cap \{x_n \leq \beta_1\}$. The claim for $x_n < 0$ follows by a similar argument, see, e.g., $I_{2^{n - 1} + 1} \cap \{x_n \geq \beta_2\}$ and $I_{3 \cdot 2^{n - 2} + 1} \cap \{x_n \geq \beta_2\}$.  
	For given a convex lattice set $K$, we obtain $L$ by a symmetry transformation of $K$ with respect to the $x_n$-axis. 
\end{remark}

\bigskip

Denote by $bd(K)$ the boundary of $K$. For $\mathrm{conv}(K \cup L)$, it is generated by the intersection of hyperplanes, and there exist facets parallel to the $x_n$-axis in $bd(\mathrm{conv}(K \cup L) )$.
More specifically, the facets in $bd(\mathrm{conv}(K \cup L) )\backslash \{bd(K), bd(L)\}$ lie in hyperplanes parallel to the $x_n$-axis. This result is summarized in the following theorem.

\begin{theorem}\label{key8}
	If $K,L \in \mathscr{C}_{\mathbb{Z}}^n$, the coefficients of hyperplanes enclosing $K$ and $L$ satisfy Table \ref{Fig:figure16}, then all the facets in $bd(\mathrm{conv}(K \cup L) )\backslash \{bd(K), bd(L)\}$ lie in hyperplanes parallel to the $x_n$-axis.
\end{theorem}

\begin{proof}
	Let $M_0$ be the intersection of $\mathrm{conv}(K \cup L)$ with the positive part of $x_n$-axis and $N_0$ be the intersection of $\mathrm{conv}(K \cup L)$ with the negative part of $x_n$-axis, and then, by the definitions of $K$ and $L$, we see that the coordinates of $M_0$ and $N_0$ are $(0, \cdots, 0, \beta_1)$ and $(0, \cdots, 0, \beta_2)$, respectively. Obviously, by the definitions of $K$ and $L$, it is easy to get that $M_0, N_0 \in K^v$ and $M_0, N_0 \in L^v$. 
	
	Now suppose $S,T \in [\mathrm{conv}(K \cup L)]^v$. 
	Without loss of generality, we assume that $S \in K^v \backslash \{M_0, N_0\}$ with the coordinate $(a_1, \cdots, a_{n - 1}, a_{n})$, and $T \in L^v \backslash \{M_0, N_0\}$ with the coordinate $(k_1, \cdots, k_{n - 1}, k_{n})$, where $a_t, k_t \in \mathbb{Z}$ for $1 \leq t \leq n$. 
	
	For $K \nsubseteq L$ or $L \nsubseteq K$, from the definitions of $K$ and $L$, we conclude that $S \neq T$, and then, there exist other faces belonging to $bd(\mathrm{conv}(K \cup L))$, see Fig. \ref{Fig:figure14}.
	
	For $S$, according to sign characteristics of the coefficients of hyperplanes (\ref{a11}) in Table \ref{Fig:figure16}, we have
	\begin{align*}
		b_{i 1} a_1 + b_{i 2} a_2 + \cdots + b_{i (n - 1)} a_{n - 1} + \beta_1 = a_n, \ i = 1, \cdots, \frac{m_1}{2}, \\
		d_{l 1} a_1 + d_{l 2} a_2 + \cdots + d_{l (n - 1)} a_{n - 1} + \beta_2 = a_n, \ l = 1, \cdots, \frac{m_1}{2}.
	\end{align*}
	
	By summing over $i$, we obtain 
	\begin{align}\label{a13}
		\sum_{i = 1}^{\frac{m_1}{2}} \sum_{j = 1}^{n - 1} b_{i j}a_j + \frac{m_1}{2} \beta_1 = \frac{m_1}{2}a_n.
	\end{align}
	Similar to (\ref{a13}), for $l$, we get that
	\begin{align}\label{a14}
		\sum_{l = 1}^{\frac{m_1}{2}} \sum_{j = 1}^{n - 1} d_{l j}a_j + \frac{m_1}{2} \beta_2 = \frac{m_1}{2}a_n.
	\end{align}
	
	Together with (\ref{a13}) and (\ref{a14}), we have
	\begin{align}\label{a17}
		\sum_{i, l = 1}^{\frac{m_1}{2}} \sum_{j = 1}^{n - 1} (b_{i j} - d_{l j})a_j + \frac{m_1}{2}(\beta_1 - \beta_2) = 0.
	\end{align}
	
	The claim for $T$ follows a similar argument
	\begin{align}\label{a18}
		\sum_{u,v= 1}^{\frac{m_1}{2}} \sum_{j = 1}^{n - 1}(c_{u  j} - e_{v j} )e_j + \frac{m_1}{2} (\beta_1 - \beta_2) = 0.
	\end{align}
	
	For $K$, by Table \ref{Fig:figure16}, we have
	\begin{align*}
		b_{i j} - d_{l j} = b_{j 2} - b_{j 1} \  {\rm{or}} \ b_{i j} - d_{l j} = b_{j 1} - b_{j 2}
	\end{align*}
	in $I_t \cup I_{2^{n - 1} + t}$ for $1 \leq t \leq 2^{n - 1}$. 
	
	\textbf{Case 1:} $b_{i j} - d_{l j} = b_{j 2} - b_{j 1} $.
	Then, by Table \ref{Fig:figure16}, we have
	\begin{align}\label{a20}
		c_{u j} - e_{v j} = - b_{j 1} - (- b_{j 2}) = b_{j 2} - b_{j 1}.
	\end{align}
	
	\textbf{Case 2:} $b_{i j} - d_{l j} = b_{j 1} - b_{j 2}$.
	It follows from Table \ref{Fig:figure16} that
	\begin{align}\label{a21}
		c_{u j} - e_{v j} = -b_{j 2} - (-b_{j 1}) = b_{j 1} - b_{j 2}.
	\end{align}
	
	For $L$, by Table \ref{Fig:figure16}, it is easy to get that
	\begin{align}\label{a22}
		c_{u j} - e_{v j} = b_{j 2} - b_{j 1} \ {\rm{or}} \ c_{u j} - e_{v j} = b_{j 1} - b_{j 2}
	\end{align}
	in $I_t \cup I_{2^{n - 1} + t}$ for $1 \leq t \leq 2^{n - 1}$.
	
	It follows from (\ref{a20}), (\ref{a21}) and (\ref{a22}) that
	\begin{align}\label{a19}
		b_{i j} - d_{l j} = c_{u j} - e_{v j}
	\end{align}
	in $I_t \cup I_{2^{n - 1} + t} $ for $ 1 \leq t \leq 2^{n - 1}$.
	
	Then, by (\ref{a17}), (\ref{a18}) and (\ref{a19}), we have
	\begin{align*}
		\sum_{i,l = 1}^{\frac{m_1}{2}} \sum_{j = 1}^{n - 1} (b_{i j} - d_{l j}) (a_j - e_j) = 0
	\end{align*}
	for $1 \leq i, l \leq \frac{m_1}{2}$, $1 \leq j \leq n - 1$.
	
	Consequently, we have
	\begin{align*}
		\sum_{j = 1}^{n - 1} \pm(b_{j 2} - b_{j 1}) (a_j - e_j) = 0.
	\end{align*}
	
	Note that $b_{j 2} - b_{j 1} \neq 0$. Then, we see that $\overrightarrow{ST}$ satisfies 
	\begin{align}\label{a23}
		\sum_{j = 1}^{n - 1} \pm(b_{j 2} - b_{j 1})x_j = 0
	\end{align}
	with $S, T \in I_t \cup I_{2^{n - 1} + t}$ for $ t = 1, \cdots, 2^{n - 1}$, $j = 1, \cdots, n - 1$.
	
	It follows from the definitions of $K,L$ and (\ref{a23}) that the $\overrightarrow{ST}$ lies on the hyperplanes strictly parallel to the $x_n$-axis.
	That is, all the facets in $bd(\mathrm{conv}(K \cup L))\backslash \{bd(K), bd(L)\}$ lie in hyperplanes parallel to the $x_n$-axis.
	
\end{proof}

The cross-polytope $K$ is defined by
$$K = \mathrm{conv}\{(\pm 1,0, \cdots, 0), \cdots, (0, \cdots, 0, \pm 1)\}.$$
$K$ is not only a $V$-polytope, but also a $H$-polytope, i.e., 
$$\{x \in \mathbb{E}^n :  |x_1| + \cdots + |x_n| \leq 1  \}.$$
Obviously, we have $K \in \mathscr{C}_{\mathbb{Z}}^n$.

In the following, a theorem characterizes the cross-polytope with the discrete Mahler product.

\begin{theorem}\label{third1}
	If $K \in \mathscr{C}_{\mathbb{Z}}^n$ is origin-symmetric, then $K$ is a cross-polytope if and only if its discrete Mahler product is the smallest.
\end{theorem}

\begin{proof}
	
	Assume that $K$ is origin-symmetric and bounded by 
	\begin{equation*}
		\begin{aligned}
			h_{\beta_1 i}(x) & = b_{i 1}x_1 + \cdots + b_{i (n-1)}x_{n-1} + \beta_1, \ {\rm{for}} \ \beta_1 > 0, i = 1, \cdots, m_1, \\
			h_{\beta_2 i}(x) & = d_{i 1}x_1 + \cdots + d_{i (n-1)}x_{n-1} + \beta_2, \ {\rm{for}} \ \beta_2 < 0, i = 1, \cdots, m_1,
		\end{aligned}
	\end{equation*}
	with $b_{i j}, d_{i j} ,\beta_1, \beta_2 \in \mathbb{Z}$ for $j = 1, \cdots, n - 1$.
	
	Then, for $K$, we have
	\begin{align*}		
		\beta_1 = -\beta_2  \geq 1, \ |b_{i j} | = |d_{i j}| \geq 1, 
	\end{align*}
	for $i = 1, \cdots, m_1$, $j = 1, \cdots, n - 1$.
	
	Consider the cross-polytope $K$ in $\mathbb{E}^n$ given by 
	$$ \{x \in \mathbb{E}^n :  |x_1| + \cdots + |x_n| \leq 1  \}.$$
	Then, in particular, for $K$, we have 
	\begin{align*}
		\beta_1 = -\beta_2  = 1, \
		|b_{i j} | = |d_{i j} | = 1, 
		\ \#(K) = 2n + 1, 
	\end{align*}
	and then, we conclude that
	\begin{equation}\label{conv4}
		\begin{aligned}
			K_{\mathbb{Z}^n}^* 
			& = \mathrm{conv}\{(\pm 1, \cdots, \pm 1, \pm 1), \cdots, (\pm 1, \cdots, \pm 1, \pm 1 ) \}\cap \mathbb{Z}^n.
		\end{aligned}
	\end{equation}
	
	By (\ref{conv4}), we have 
	$$\#(K_{\mathbb{Z}^n}^* ) = \#([-1,1]^n \cap \mathbb{Z}^n) = 3^n,$$ 
	and therefore, for $K$, the discrete Mahler volume is given by
	$$\#(K) \#(K_{\mathbb{Z}^n}^*) = (2n + 1) 3^n.$$
	
	Assume that $K_1 \in \mathscr{C}_{\mathbb{Z}}^n$ and $K_1$ is origin-symmetric, and then, we let $t = \beta_1 = -\beta_2 \geq 1$ and $|b_{i j}| = |d_{i j}| \geq 1$, where $b_{i j}, d_{i j}, \beta_1, \beta_2 \in \mathbb{Z}$. Clearly, when $t = 1$ and $|b_{i j}| = |d_{i j}| = 1$, we have $K_1 = K$. We now separate the proof into three cases.
	
	\textbf{Case 1:} $|b_{i j} |= |d_{i j} | \geq 2, t = \beta_1 = -\beta_2 \geq 2.$
	
	Obviously, we have $K \subset K_1$, and then, $\#(K) < \#(K_1)$. 
	Therefore, a lower bound for $\#((K_1)_{\mathbb{Z}^n}^*)$ is given by
	$$\#((K_1)_{\mathbb{Z}^n}^*) \geq \#([-2,2]^n \cap \mathbb{Z}^n) = 5^{n}.$$
	By the definition of $K_1$, the upper bound of $(K_1)_{\mathbb{Z}^n}^*$ is an $n$-dimensional cube with edge length $2t$. Then, we can see that
	\begin{align*}
		\#((K_1)_{\mathbb{Z}^n}^*) \leq \#([-t, t]^n \cap \mathbb{Z}^n) = (2t + 1)^n,
	\end{align*}
	and thus, we get that 
	\begin{align*}
		5^{n}  \leq \#((K_1)_{\mathbb{Z}^n}^*) \leq (2t + 1)^n.
	\end{align*}
	
	Therefore, for any $n \geq 1$, we have 
	\begin{align*}
		\frac{\#(K_1)\#((K_1)_{\mathbb{Z}^n}^*)}{\#(K) \#(K_{\mathbb{Z}^n}^*)} > \frac{(2n + 1) 5^n}{(2n + 1) 3^n} 
		= \frac{5^n}{ 3^n} > 1,
	\end{align*}
	which implies that $\#(K_1) \#((K_1)_{\mathbb{Z}^n}^*) > \#(K) \#(K_{\mathbb{Z}^n}^*).$
	
	\textbf{Case 2:} $|b_{i j} | = |d_{i j} | \geq 2, t = \beta_1 = -\beta_2 = 1.$
	
	When $ t = \beta_1 = -\beta_2 = 1$, for $K_1 \in \mathscr{C}_{\mathbb{Z}}^n$ and $K_1$ is origin-symmetric, we have $|b_{i j} | = |d_{i j} | \leq 1$, a contradiction.
	
	\textbf{Case 3:} $|b_{i j} | = |d_{i j} | = 1, t = \beta_1 = -\beta_2 \geq 2.$
	
	By the definitions of $K$ and $K_1$, we have $K \subset K_1$, and thus, we get that $\#(K) < \#(K_1)$. For $K_1$, we conclude that
	\begin{align*}
		(K_1)_{\mathbb{Z}^n}^* = \mathrm{conv}\{(\pm 1, \cdots, \pm 1, \pm t), \cdots, (\pm 1, \cdots, \pm 1, \pm t)\} \cap \mathbb{Z}^n,
	\end{align*}
	and then, the upper and lower bounds for $\#((K_1)_{\mathbb{Z}^n}^*)$ are given by
	\begin{align*}
		3^n 2 =  \#([-1,1]^n \cap \mathbb{Z}^n) 2 \leq \#((K_1)_{\mathbb{Z}^n}^*) \leq  \#([-1,1]^n \cap \mathbb{Z}^n) t =  3^n t.
	\end{align*}
	
	Thus, for any $n \geq 1$, we have 
	\begin{align*}
		\frac{\#(K_1)\#((K_1)_{\mathbb{Z}^n}^*)}{\#(K) \#(K_{\mathbb{Z}^n}^*)} > \frac{2 (2n + 1) 3^n}{(2n + 1) 3^n} = 2 > 1,
	\end{align*}
	and then, we get that $\#(K_1) \#((K_1)_{\mathbb{Z}^n}^*) > \#(K) \#(K_{\mathbb{Z}^n}^*).$
	
	By the above argument, $K$ is a cross-polytope if and only if its discrete Mahler product is the smallest.
	
	As shown above, the smallest discrete Mahler product can be calculated by 
	\begin{align*}
		\#(K)\#(K_{\mathbb{Z}^n}^*) = (2n + 1) 3^n.
	\end{align*}
\end{proof}



\begin{thebibliography}{99}
	
	
	
	
	\bibitem{article19}Adivar, M., Fang, S.C.: Convex Analysis and Duality over Discrete Domains. J. Oper. Res. Soc. China. 6, 189-247 (2018)
	
	
	
	
	
	
	
	
	
	
	\bibitem{article24}Artstein-Avidan, S., Karasev, R., Ostrover, Y.: From Symplectic Measurements to the Mahler Conjecture. Duke Math. J. 163, 2003-2022 (2014)
	
	
	\bibitem{book4}Barvinok, A.: A Course in Convexity. AMS, Providence, Rhode Island (2002)
	
	\bibitem{article22}Barthe, F., Fradelizi, M.: The volume product of convex bodies with many hyperplane symmetries. Am. J. Math. 135, 311-347 (2013)
	
	
	\bibitem{article27}Berg, S.L., Henk, M.: Lattice Point Inequalities for Centered Convex Bodies. SIAM J. Discrete Math. 30, 1148-1158 (2016)
	
	\bibitem{article1}Bourgain, J., Milman, V.D.: New volume ratio properties for convex symmetric bodies in $\mathbb{R}^n$. Invent. Math. 88, 319-340 (1987)
	
	\bibitem{article23}Fradelizi, M., Hubard, A., Meyer, M., Roldan${\rm \acute{a}}$n-Pensado, E., Zvavitch, A.: Equipartitions and Mahler volumes of symmetric convex bodies. Am. J. Math. 144, 1201-1209 (2022)
	
	\bibitem{article15}Gardner, R.J., Gritzmann, P.: Discrete tomography: Determination of finite sets by $X$-rays. Trans. Am. Math. Soc. 349, 2271-2295 (1997)
	
	\bibitem{article14}Gardner, R.J., Gronchi, P.: A Brunn-Minkowski Inequality for the Integer Lattice. Trans. Am. Math. Soc. 353, 3995-4024 (2001)
	
	\bibitem{article28}Gardner, R.J., Gronchi, P., Zong, C.: Sums, Projections, and Sections of Lattice Sets, and the Discrete Covariogram. Discrete Comput. Geom. 34, 391-409 (2005)
	
	
	\bibitem{article4}Gordon, Y., Meyer, M., Reisner, S.: Zonoids with minimal volume product - a new proof. Proc. Am. Math. Soc. 104, 273-276 (1988)
	
	\bibitem{article20}Iriyeh, H., Shibata, M.: Minimal Volume Product of Three Dimensional Convex Bodies with Various Discrete Symmetries. Discrete Comput. Geom. 68, 738-773 (2022)
	
	
	\bibitem{book1}Hug, D., and Weil, W.: Lectures on Convex Geometry. Springer, Switzerland (2020)
	
	\bibitem{article3}Lin, Y., Leng, G.: Convex bodies with minimal volume product in $\mathbb{R}^2$ - a new proof. Discrete Math. 310, 3018-3025 (2010)
	
	\bibitem{article9}Lopez, M.A., Reisner, S.: A Special Case of Mahler's Conjecture. Discrete Comput. Geom. 20, 163-177 (1998)
	
	\bibitem{article12}Maehara T., Murota, K.: A framework of discrete DC programming by discrete convex analysis. Math. Program. 152, 435-466 (2015)
	
	\bibitem{book5}Matou$\check{\rm{s}}$ek, J.: Lectures on Discrete Geometry. Springer, New York (2002)
	
	\bibitem{article8}Murota, K.: Discrete Convex Analysis. Math. Program. 83, 313-371 (1998)
	
	
	\bibitem{article21}Fradelizi, M., Meyer, M., Zvavitch, A.: An Application of Shadow Systems to Mahler's Conjecture. Discrete Comput. Geom. 48, 721-734 (2012)
	
	
	
	\bibitem{article7}Nazarov, F., Petrov, F., Ryabogin, D., Zvavitch, A.: A remark on the Mahler conjecture : local minimality of the unit cube. Duke Math. J. 154, 419-430 (2010)
	
	\bibitem{article13}Reisner, S.: Minimal volume-product in Banach spaces with a 1-unconditional basis. J. London Math. Soc. 36, 126-136 (1987)
	
	
	\bibitem{book3}Schneider, R.: Convex Bodies: The Brunn-Minkowski Theory, 2nd expanded edn. Cambridge University Press, Cambridge (2014)
	
	\bibitem{article2}Wang, J., Si, L.: The Polar of Convex Lattice Sets. Contrib. Discrete Math. 15, 42-51 (2020)
	
	\bibitem{article31}Ziegler, G.M.: Lectures on Polytopes, Revised 6th printing. Springer, Berlin (2006)
	
	\bibitem{article30}Zong, C.: The cube: a window to convex and discrete geometry. Cambridge University Press, Cambridge (2006).
	
	
	
	
	
	
	
	
	
	
	
\end{thebibliography}

\renewcommand\refname{Reference}

\vskip 20pt

\noindent Tingting He

\noindent College of Science,
Beijing Forestry University, Beijing 100083,
P.R.China

\noindent E-mail: moment20212021@163.com

\vskip 5pt

\noindent Lin Si

\noindent College of Science,
Beijing Forestry University, Beijing 100083,
P.R.China

\noindent E-mail: silin@bjfu.edu.cn

  \end{document}